\documentclass[12pt]{article}
\usepackage{amsmath}
\usepackage{amssymb}
\usepackage{amsthm}
\usepackage[hidelinks]{hyperref}
\usepackage{graphicx}
\usepackage{CJKutf8}
\usepackage{tikz-cd}
\usepackage{enumerate}
\usepackage{subcaption}
\usepackage{appendix}
\usepackage{ulem}
\usepackage{mathrsfs}  
\newtheorem{theorem}{Theorem} [section]
\newtheorem*{theorem*}{Theorem}
\newtheorem{proposition}[theorem]{Proposition}
\newtheorem{remark}[theorem]{Remark}

\newtheorem{definition}[theorem]{Definition}
\newtheorem{lemma}[theorem]{Lemma}
\newtheorem{corollary}[theorem]{Corollary}

\newtheorem{claim}{Claim}

\usepackage{breakcites}
\definecolor{c1}{HTML}{7e0f12}

\newcommand{\RR}{\mathbb{R}}  
\newcommand{\HH}{\mathcal{H}}
\newcommand{\MM}{\mathcal{M}}
\newcommand{\NN}{\mathcal{N}}
\newcommand{\KK}{\mathcal{K}}

\newcommand{\UU}{\mathcal{U}}

\newcommand{\dist}[1]{\operatorname{dist}({#1})}

\newcommand{\spt}[1]{\operatorname{spt} {#1} }
\newcommand{\inte}[1]{\operatorname{Int}_{\bar{\Omega}}\left({#1}\right)}

\newcommand{\noopsort}[1]{}

\begin{document}
\title{Some Foundational Results for Free Boundary Brakke Flows}
\author{Yueheng Bao}

\maketitle

\begin{abstract}
In this paper, we establish some geometric and analytic foundations for free boundary Brakke flows. Specifically, we (i) introduce  unit-regular and cyclic free boundary flows and show that they are preserved under reflections and weak limits, (ii) prove that the support of free boundary Brakke flows satisfies an avoidance principle, and (iii) introduce free boundary inner and outer flows and prove the existence of matching free boundary Brakke flows. These results serve as general tools to analyze free boundary flows through singularities, and in particular will be applied in forthcoming work with Haslhofer, where we address the mean-convex neighborhood conjecture and uniqueness conjecture for free boundary flows through (half) cylindrical singularities.
\end{abstract}

\tableofcontents 

\section{Introduction}
In this paper, we fix a smooth open domain $\Omega \subset \RR^{n+1}$. 
For convenience, we will write $S=\partial \Omega$, which is orientable and smooth. 
Free boundary refers to an added constraint that a hypersurface $M$ meets the barrier $S$
 orthogonally at its boundary, similar to the way bubbles touch a glass window. As defined in \cite[Def 1.1]{stahl96} and \cite[Def 1.1]{buckland05}, a smooth free boundary mean curvature flow in $\Omega$ is a family of smooth embeddings $\{M_t\subset \bar{\Omega}\}_{t\in I}$ satisfying the mean curvature flow (MCF) equation
$$\partial_t x=H(x),\ x\in M_t \cap \Omega,$$
and the Neumann boundary condition 
$$\nu_{M_t}(x) \perp \nu_S(x),\ x\in M_t\cap S,$$ 
where $\nu_{M_t}$ is a choice of the unit normal of the hypersurface $M_t$ and $\nu_S$ is the outwards pointing unit normal of the boundary $S$.

Stahl \cite{stahl96} proved long-time existence of the smooth, compact, free-boundary mean curvature flow of hypersurfaces unless curvature blow-up occurs at a finite time singularity. Buckland \cite{buckland05} derived a special monotonicity formula to prove type-I singularities are modeled on (half) shrinkers. 
Edelen \cite{edelen16} proved that certain type-II singularities can be realized by translating solitons. \\

To properly study the flow through singularities, one needs weak solutions.
A free boundary integral Brakke flow in $\Omega$ is given by a family of Radon measures $\{\mu_t\}_{t\in I}$ supported in $\bar{\Omega}$, such that 
\begin{equation}\label{equ:brakke}
\int \phi(\cdot, t)d\mu_t\bigg|_{t=a}^b\le \int_a^b \int\left( -|S^T H|^2\phi+S^TH\cdot D\phi+\partial_t\phi \right)d\mu_tdt
\end{equation}
for all nonnegative test functions $\phi$ with $D_{\nu_S}\phi=0$. Here, $S^TH=H-1_S(H\cdot \nu_S)\nu_S$, and at almost every time associated to $\mu_t$ there is an integral varifold $V_t$ with first variation 
$$\delta V_t(X)=-\int S^TH\cdot X d\mu_t$$
for all $X$ tangential to $\partial \Omega$. So, in particular $V_t$ is a free boundary varifold in the sense of \cite{gruterjost} and \cite{demasi21}. 

Free boundary Brakke flows have been introduced first by Mizuno-Tonegawa \cite{mizunotonegawa15,err-mt15} and Kagaya \cite{kagaya18}, who constructed them as sharp interface limits of the Allen-Cahn equations with Neumann boundary. In an important paper \cite{edelen18}, Edelen systematically developed the theory of free boundary Brakke flows, including in particular a compactness theorem, a notion of reflected Gaussian density $\Theta_\mathrm{refl}$ and the corresponding $\varepsilon$-regularity theorem, and existence via elliptic regularization.
\\

Another important notion of weak solutions are level set solutions.
Similar to the case without boundary, a free boundary subsolution of the MCF is defined to be a family of closed sets satisfying the avoidance principle with all smooth compact free boundary MCFs, that is:
If $\left\{N_t \subset \bar{\Omega} \right\}_{t \in\left[t_0, t_1\right]}$ is a smooth compact free boundary MCF with $M_{t_0} \cap N_{t_0}=\emptyset$, then $M_t \cap N_t=\emptyset$ for all $t \in\left[t_0, t_1\right]$. 
Apparently, the choice of subsolution is highly nonunique, thus one defines the free boundary level set flow $F_t(M)$ to be the maximal one, i.e.,
\begin{equation}\label{equ:levelsetflow}
F_{t}(M)=\overline{\bigcup \{M_{t}|\{M_{t'}\}_{t'\in [0,+\infty)}\text{ is a subsolution with } M_0=M\}.}
\end{equation}
The level set flow with free boundary has been explored first by Giga-Sato \cite{gs93} and Sato \cite{sato94}, where the theory of viscosity solutions was used to solve the level set MCF equation
$$
  u_t = |Du| \, \mathrm{div}\!\left( \frac{Du}{|Du|} \right)
  \quad \text{in } \Omega \times (0,T),
$$
with the Neumann boundary condition
$
  D_{\nu_{S}}u = 0
  \quad \text{on } \partial\Omega \times (0,T),
$
and any continuous initial condition.
\\

For the free boundary flow of mean-convex hypersurfaces, which can be described either as Brakke solution or as level set solution, a precise regularity and structure theory, paralleling the one by White \cite{whi-size, whi-nature} has been developed by Edelen-Haslhofer-Ivaki-Zhu \cite{ehiz22}. On the other hand, several recent developments for surfaces without boundary, in particular the solution of the mean-convex neighborhood conjecture \cite{chh,chhw,bl25a,bl25b},
 the proof of the genericity conjecture \cite{ccms_inven,ccms_duke,ccs,ccms_crel},
and the multiplicity-one conjecture \cite{bk23}, naturally take place in the framework of unit-regular, cyclic, integral Brakke flows associated to outer/inner flows, which has not been available yet for flows with boundary. The purpose of the present paper is to establish such geometric and analytic foundations for flows with free boundary. \\

\textbf{Main results.}
Our first results concern new notions of unit-regular and cyclic free boundary flows, and their preservation under limits and reflections. To discuss this, recall that in the setting without boundary, Schulze and White \cite{sw20} (see also White \cite{whi05localreg}) introduced a notion of unit-regular flows that prevents sudden vanishing of smooth points. Specifically, they called a Brakke flow unit-regular if it is smooth in a two-sided parabolic neighborhood of every space-time point with Gaussian density one, and proved that this class of flows is preserved under weak limits. Moreover, White \cite{whi09currents} introduced a notion of cyclic flows that captures the fact that the hypersurface bounds an inside and an outside, and proved that this class is preserved under weak limits as well. Here, we generalize these notions to the setting with boundary as follows:
\begin{definition}[unit-regular and cyclic]\label{def:regcyc}
A free boundary integral Brakke flow $\MM=\{\mu_t\}_{t\in I}$ in a domain $\Omega$ is called 
\begin{enumerate}[(1)]
\item unit-regular, if every space-time point of reflected Gaussian density one is a regular point, i.e. for all $X=(x, t)$ with $\Theta_\mathrm{refl}(\MM,X)=1$, there exists an $\varepsilon=\varepsilon(X)>0$ such that $t^{\prime} \mapsto \operatorname{spt}\left(\mu_{t^{\prime}}\right) \cap$ $B(x, \varepsilon)$ is a smooth mean curvature flow (possibly with free boundary) for $t^{\prime} \in\left[t-\varepsilon^2, t+\varepsilon^2\right]$.
\item cyclic, if for almost all $t$ the associated $\mathbb{Z}_2$ flat chain $\left[V_t\right]$ is relatively cyclic, namely $\operatorname{spt}(\partial[V_t])\subset S$ for a.e. $t$. Here, $[V_t]$ is obtained from $V_t$ by reducing the coefficients modulo 2, and $\partial$ denotes the boundary operator of $\mathbb{Z}_2$ flat chains.
\end{enumerate}
\end{definition}
Note that if $\MM$ is supported in the interior of $\Omega$, then $\Theta_\mathrm{refl}(\MM, X)=\Theta(\MM, X)$ (c.f. [Hui90]) and $\partial [V_t]=\emptyset$, so our definitions reduce to the ones from \cite{sw20} and \cite{whi09currents}, respectively.
We first prove that our generalized properties are also preserved under weak limits:

\begin{theorem}[preservation under weak limits]\label{thm:limitregcyc}
Unit-regular free boundary integral Brakke flows are closed under weak convergence of free boundary Brakke flows. Likewise are cyclic free boundary integral Brakke flows. 
\end{theorem}

In particular, taking also into account existence via elliptic regularization (c.f. Theorem \ref{thm:hwcpamB6} below), our theorem applies for blow-up sequences. Specifically, given a unit-regular, cyclic, free boundary integral Brakke flow $\MM=\{\mu_t\}$ in a domain $\Omega_0$, a sequence of space-time points $X_i=(x_i,t_i)\in \bar{\Omega}_0\times \RR_+$, and rescaling factors $\lambda_i\to \infty$, one considers the sequence of flows $\MM^{i}=\mathcal{D}_{\lambda_i}(\MM-X_i)$ that is obtained from $\MM$ by shifting $X_i$ to the origin and parabolically rescaling by $\lambda_i$. Our theorem guarantees that any blowup limit $\MM^{\infty}=\lim\limits_{i\to\infty}\MM^i$ is unit-regular and cyclic. In particular, we obtain:

\begin{corollary}[ruling out certain blow up limits]\label{cor:blowupnotbad}
For any free boundary integral Brakke flows that are unit-regular and cyclic, quasistatic multiplicity one (half) planes or  $Y\times \RR^{n-1}$ or $Y\times \RR^{n-1}_+$ cannot occur as potential blow up limits. Here, $Y$ refers to the union of three rays meeting at $120^\circ$ angles.
\end{corollary}

Note that for blowups near the boundary, specifically if 
$\lambda_i d(x_i,\partial \Omega_0)\le C$, the limit $\MM^{\infty}$ is a free boundary flow in the halfspace $\Omega_{\infty}=\RR^{n+1}_{+}$. Edelen \cite{edelen18} observed that such $\MM^{\infty}$ can be reflected to obtain an integral Brakke flow without boundary $\tilde{\MM}^{\infty}$ in $\RR^{n+1}$. Here, we show that our notions of unit-regular and cyclic behave well under reflections:
\begin{theorem}[preservation under reflections]\label{thm:reflectprop}
Given an integral, unit-regular, cyclic free boundary Brakke flow in the halfspace $\mathbb{R}^{n+1}_+$, the reflected flow is also an integral, unit-regular, cyclic flow (without boundary) in $\mathbb{R}^{n+1}$.
\end{theorem}
In particular, the theorem shows that our notions from Definition \ref{def:regcyc} are compatible with the ones from \cite{sw20} and \cite{whi09currents} under Edelen's reflection. Moreover, as will be discussed in more detail in forthcoming work, Theorem \ref{thm:reflectprop} provides a tool to analyze boundary singularities via classification results for ancient unit-regular, cyclic, integral Brakke flows in entire space, e.g. the ones from \cite{chh,chhw,bl25b,ccms_inven,ccs}.

Out next result is an avoidance principle for free boundary Brakke flows. To discuss this, first recall that by the maximum principle, smooth MCFs do not bump into each other, namely if they are disjoint at some time $t_0$, then they remain disjoint at all later times.\footnote{Here we assume that at least one of the solutions is compact, though this assumption could be relaxed by \cite{whi24ap}.} This was generalized by Ilmanen \cite{ilm94}, who proved that the support of a Brakke flow satisfies the avoidance principle when compared with any smooth compact solution. Here, we establish a free boundary version of Ilmanen's avoidance principle, at least for convex domains:
\begin{theorem}[avoidance principle]\label{thm:0elliptic10.5}
Let $\mathcal{M}=\{\mu_t\}$ be a free boundary integral Brakke flow in a convex domain $\Omega$. Then for any smooth compact free boundary flow $\{N_t\}_{t\in [t_0,t_1]}$ with $(\spt {\mathcal{M}})_{t_0}\cap N_{t_0}=\emptyset$, we have $(\spt {\mathcal{M}})_{t}\cap N_{t}=\emptyset$ for all $t\ge t_0$. 
\end{theorem}
Here, $(\spt{\mathcal{M}})_t$ denotes the support at time $t$, namely the set of all $x\in \bar{\Omega}$ with $\Theta_\mathrm{refl}(\MM, (x,t))\ge 1$. In the special case when the flows have empty boundary, our avoidance principle of course recovers the one by \cite{ilm94}. Another way to phrase the conclusion of Theorem \ref{thm:0elliptic10.5} is that the support of $\MM$ is a subsolution of the free boundary level set flow. It has been shown by Sato \cite{sato94} that subsolutions also avoid each other, and thus as an important corollary, we obtain: 
\begin{corollary}[avoidance principle]
Let $\MM, \NN$ be compact free boundary integral Brakke flows in a convex domain. If $(\spt{\MM})_{t_0}\cap (\spt{\NN})_{t_0}=\emptyset$, then we have $(\spt {\mathcal{M}})_{t}\cap (\spt{\NN})_{t}=\emptyset$ for all $t\ge t_0$. 
\end{corollary}

\vspace{\baselineskip}

Finally, using Theorem \ref{thm:limitregcyc} and Theorem \ref{thm:0elliptic10.5}, we will match a unit-regular, cyclic free boundary integral Brakke flow to the free boundary outer/inner level set flow. To discuss this, recall that in the setting without boundary, 
Hershkovits and White \cite{hw20} discussed the outer and inner flow for an initial embedded hypersurface, and proved  that one can construct a Brakke flow whose space-time support exactly matches the outer/inner flow. 
Generalizing the notions from \cite{hw20}, we will now define outer and inner flows for the free boundary case.  Let $K\subseteq \bar{\Omega}$ be a closed subset and assume that $M:=\partial K:=K\backslash \inte{K}$ is a smooth free boundary surface, where $\inte{K}$ is the interior of $K$ viewed as a subset of topological space $\bar{\Omega}$. 
Denote $K^{\prime}:=\overline{\bar{\Omega}\backslash K}$, and evolve $K$ and $K^{\prime}$ by free boundary level set flow, c.f. (\ref{equ:levelsetflow}).
We consider the outer and inner space time tracks
$$
\begin{aligned}
\mathcal{K} & :=\left\{(x, t) \subseteq \bar{\Omega}\times [0,+\infty) \mid x \in F_t(K)\right\} \\
\mathcal{K}^{\prime} & :=\left\{(x, t) \subseteq \bar{\Omega}\times [0,+\infty) \mid x \in F_t(K^{\prime})\right\}.
\end{aligned}
$$
\begin{definition}[outer and inner free boundary flow]
Setting $\partial\mathcal{K}:=\mathcal{K}\backslash \operatorname{Int}_{\bar{\Omega}\times \RR}\left(\mathcal{K}\right)$ and $\partial\mathcal{K}^{\prime}:=\mathcal{K}^{\prime}\backslash \operatorname{Int}_{\bar{\Omega}\times \RR}\left(\mathcal{K}^{\prime}\right)$, 
we call
\begin{equation}\label{equ:hwcpam2}
\begin{split}
M_t & :=\left\{x \in \bar{\Omega} \mid(x, t) \in \partial \mathcal{K}\right\} \\
M^{\prime}_t & :=\left\{x \in \bar{\Omega} \mid(x, t) \in \partial \mathcal{K}^{\prime}\right\}
\end{split}
\end{equation}
the outer and inner free boundary flows of $(K,M)$. 
\end{definition}

These flows are nontrivially separated when fattening occurs.
Using these notions, our free boundary version of the matching theorem \cite[Thm B6]{hw20} is the following:
\begin{theorem}[matching free boundary Brakke flows]\label{thm:hwcpamB6}
Given any $K\subseteq \bar{\Omega}$ with $M:=\partial K$ being a smooth compact free boundary surface, where $\Omega$ is a convex domain, set $\mu=\mathcal{H}^n\llcorner\partial K$. There exist unit-regular, cyclic free boundary integral Brakke flows $\{\mu_t\}_{t\ge0}$ and $\{\mu^{\prime}_t\}_{t\ge0}$ such that $\mu_0=\mu^{\prime}_0=\mu$ and the space-time support of the flow is the space-time set swept out by the outer flow $M_t$ and inner flow $M^{\prime}_t$, respectively.
\end{theorem}

Theorem \ref{thm:hwcpamB6} has already been used implicitly in Corollary \ref{cor:blowupnotbad} above. Moreover, Theorem \ref{thm:hwcpamB6} is an important tool to study uniqueness of free boundary flows through singularities (for uniqueness without boundary see e.g. \cite{hw20,chh,chhw,cds}). In particular, Theorem \ref{thm:hwcpamB6} will be applied in forthcoming work by Haslhofer and the author, where we address the mean-convex neighborhood conjecture and uniqueness conjecture for free boundary flows through (half) cylindrical singularities.
\\

\textbf{Outline.} 
Regarding Theorem \ref{thm:limitregcyc},  to prove that being unit-regular is preserved under weak limits, we adapt the argument from [SW, Thm 4.2] to our setting, where we now use Edelen's quantity $\Theta_{\text{refl}}$. And to prove that being cyclic is preserved under weak limits we observe that our definition is designed, such that after restricting to the interior of $\Omega$ we can apply the results from [Whi09].

Regarding Theorem \ref{thm:reflectprop}, to prove that being unit-regular is preserved under reflections, we derive an elementary PDE lemma about reflections of graphical MCFs with free boundary. Most interestingly, to prove that being cyclic is preserved under reflections, given any free boundary integral varifold $V$ in the halfspace $\RR_+^{n+1}$, we consider the doubled varifold $\tilde{V}=V+A_{\#}V$, where $A$ denotes the reflection across the hyperplane $S=\partial \RR_+^{n+1}$, and establish the crucial identity 
\begin{equation}\label{eq:star}
\partial [\tilde{V}]=\partial [V]+A_{\#}\partial [V]. 
\end{equation}
Observing also that $\spt{\partial [V]}\subset S$ and $A_{\#}|_S=id$, we infer that the boundary of $[\tilde{V}]$ always occurs with even multiplicity, hence vanishes modulo 2. To prove (\ref{eq:star}), we first establish the corresponding formula at the level of polyhedral chains and then conclude by approximation.

To prove our avoidance principle for free boundary  Brakke flows (Theorem \ref{thm:0elliptic10.5}), we first extend the class of test functions in the definition of free boundary Brakke flows. Specifically, via an approximation argument, we show that (\ref{equ:brakke}) actually holds even without assuming $D_{\nu_S}\phi=0$. 
Using this, we can then adapt the argument of Ilmanen \cite[Sec 10.5]{ilm94} (see also \cite[Thm 34]{hw23}) to our setting, and insert, loosely speaking,  the test function $\phi(x,t)=(\gamma-d(x,t))_+^3$ in (\ref{equ:brakke}), where $d(x,t)$ denotes the distance to the smooth comparison flow $\{N_t\}_{t\in[t_0,t_1]}$, and $\gamma>0$ is a small constant.
 Most intriguingly, when applying integration by part, we encounter a new boundary term. Fortunately, using convexity we can show that $D \phi\cdot \nu_S\le 0$ on $S$, and consequently the boundary term has the good sign.

To prove the matching theorem for free boundary Brakke flows (Theorem \ref{thm:hwcpamB6}), we first check some fundamental relationship between outer/inner flows and the boundary of level set flows, where we adapt the arguments from \cite[Prop A.3, Prop B.2, Thm C.1]{hw20} and \cite[Thm 1.8]{bk23} to our setting. With the help of these, we can then adapt the argument of \cite[Prop B.6]{hw20} to our setting, where we choose a sequence of suitable approximating regions $K_i$ by a perturbation argument. Applying free boundary elliptic regularization theory, as provided by \cite[Sec 9]{edelen18}, we obtain $e^{-\frac{z}{\varepsilon}}$-weighted area  minimizing integral $(n+1)$-currents $P^{i,\varepsilon}$ supported in $\bar{\Omega}\times [0,+\infty)$ with $\partial P^{i,\varepsilon}=\partial K_i$. We define $\mu^{i,\varepsilon}_t:=\mu_{P^{i,\varepsilon}_t}$, where $P^{i,\varepsilon}_t:=P^{i,\varepsilon}-\frac{t}{\varepsilon}e_z$, and define $E^{i,\varepsilon}_t$ to be enclosed domain. Using Theorem \ref{thm:limitregcyc}, we can obtain a unit-regular and cyclic free boundary integral Brakke flow $\tilde{\mu}^i_t=\mu^i_t\times \mathcal{H}^1$ starting from $\partial K_i \times \RR$ via 
\begin{equation}\label{equ:0pairconverge}
(\mu^{i,\varepsilon}_t,E^{i,\varepsilon}_t) \stackrel{\varepsilon\to0}{\longrightarrow}(\tilde{\mu}^i_t,\tilde{E}^i_t),
\end{equation}
Here, to be more detailed compared to \cite[Thm B.6]{hw20}, using Theorem \ref{thm:0elliptic10.5} we prove that the reduced boundary satisfies $\partial^* \tilde{E}^i_t\subseteq \spt{\tilde{\mu}^i_t}$, and that $\tilde{E}^i_t=F_t(\tilde{K}_i)$ up to a null set,
where $\tilde{K}_i=K_i\times \RR\subset \Omega\times \RR$. Using this, we can conclude that the support of $\mu_t=\lim\limits_{i\to\infty} \mu_t^i$ is given by the space-time set swept out by the outer flow $\{M_t\}_{t\ge 0}$.
\\

This paper is organized as follows. We collect some preliminaries in Section 2, including the most fundamental definitions and results for free boundary mean curvature flows. In Section 3, we prove Theorem \ref{thm:limitregcyc} and Theorem \ref{thm:reflectprop}. In Section 4, we prove Theorem \ref{thm:0elliptic10.5}. And finally in Section 5, we prove Theorem \ref{thm:hwcpamB6}. 

\medskip

\textbf{Acknowledgements.} The author has been supported by a Blyth Fellowship, a Mary H. Beatty Fellowship, a Department of Mathematics Graduate Program Award and an International Graduate Student Scholarship in Mathematics from the University of Toronto, and he acknowledges his supervisor Robert Haslhofer for his invaluable guidance and support in bringing this paper into fruition.


\section{Preliminaries}

In this section, we provide some essential background about free boundary level set flows and Brakke flows that will be used frequently. 
\subsection{Free boundary level set flow}
We start with the definition of free boundary subsolutions.
\begin{definition}[free boundary subsolution, c.f. \text{\cite[Sec 3]{ehiz22}}] \label{def:setsubsol}
A family of closed sets $\left\{K_t \subset \bar{\Omega}\right\}_{t \in[0,+\infty)}$ is called subsolution of the free boundary mean curvature flow if it satisfies the following avoidance principle: \\
If $\left\{N_t \subset \bar{\Omega} \right\}_{t \in\left[t_0, t_1\right]}$ is a smooth compact free boundary mean curvature flow with $K_{t_0} \cap N_{t_0}=\emptyset$, then $K_t \cap N_t=\emptyset$ for all $t \in\left[t_0, t_1\right]$, where $0 \leq t_0 \leq t_1 < +\infty$. 
\end{definition}

\begin{definition}[free boundary level set flow]\label{def:levelsetflow}
The free boundary level set flow $F_t(K)$ is the maximal subsolution with initial condition $F_0(K)=K$.
\end{definition}

As pointed out by \cite{gs93} and \cite{ehiz22}, we have:
\begin{proposition}[basic properties]\label{prop:fbprop}
The free boundary level set flow has the following basic properties:
\begin{enumerate}[(1)]
\item Consistency with smooth flows: If $\left\{M_t\right\}_{t \in[0, T]}$ is a smooth free boundary mean curvature flow, then $M_t=F_t\left(M_0\right)$.
\item Avoidance: If $K$ and $L$ are disjoint compact sets, then $F_t(K)$ and $F_t(L)$ are also disjoint and $F_t(K \cup L)=F_t(K) \cup F_t(L)$.
\item Inclusion: If $K \subseteq L$, then $F_t(K) \subseteq F_t(L)$.
\item Semigroup property: $F_{t_1+t_2}(K)=F_{t_2}\left(F_{t_1}(K)\right)$.
\end{enumerate}
\end{proposition}

\subsection{Free boundary varifolds and Brakke flows}
Recall from \cite[Sec 4]{allard72} that the first variation of an integral varifold  $V$ is given by
$$
\delta V(X)=\int \operatorname{div}_V(X) d \mu_V, \quad X \in C_c^1\left(\mathbb{R}^{n+1}, \mathbb{R}^{n+1}\right).
$$
We will assume that $\delta V$ is locally finite, i.e.
$$
|\delta V(X)| \leq C_W|X|_{C^0} \quad \forall X \in C_c^1\left(W, \mathbb{R}^{n+1}\right), W \subset \subset \mathbb{R}^{n+1},
$$
so that we can define the total variation (Radon) measure
$$
\|\delta V\|(W)=\sup \left\{\delta V(Y): Y \in C_c^1\left(W, \mathbb{R}^{n+1}\right),|Y| \leq 1\right\} 
$$
and therefore obtain
\begin{equation}\label{equ:fvf}
\delta V(X)=-\int H \cdot X d \mu_V+\int X \cdot n_V d \sigma_V,\ \forall X \in C_c^1\left(\mathbb{R}^{n+1}, \mathbb{R}^{n+1}\right),
\end{equation}
where $H=-\frac{d\|\delta V\|}{d \mu_V}$ is the generalized mean curvature vector, $\sigma_V=\|\delta V\|_{\text {sing }}$ is the generalized boundary measure, and $n_V$ is the generalized outwards conormal. Now we can define:

\begin{definition}[free boundary varifold, \cite{edelen18}]\label{def:freebdyvfd}
An integral free boundary varifold in $\Omega$ is an integral $n$-varifold $V$ defined in $\mathbb{R}^{n+1}$ with $\operatorname{spt}(\mu_{V})\subset \bar{\Omega}$, such that
\begin{equation}\label{equ:ede5}
\delta V(X)=-\int S^T(H) \cdot X d \mu_V, \quad \forall X \in \mathcal{T}(S) \cap C_c^1
\end{equation} 
for some $S^T(H) \in L_{l o c}^1\left(\mathbb{R}^{n+1}, \mathbb{R}^{n+1}; \mu_V\right)$.
 Here $\mathcal{T}(S)$ is the space of vector fields in $\RR^{n+1}$ that lie tangential when restricted to $S$.
\end{definition}
By \cite[Prop 3.2]{edelen18}, $V$ has free boundary in $S$ if and only if $\|\delta V\|$ is locally bounded, $\operatorname{spt}(\sigma_V)\subset S$ and $n_{V}=\nu_S \text{ at }\sigma_V\text{-a.e. }x$, where $\sigma_V$ is the singular part of the measure.

The admissible test functions are set to be:
$$
\mathcal{B} \mathcal{T}(S, [a,b])=\left\{\phi \in C_c^1\left(\mathbb{R}^{n+1} \times [a,b], \mathbb{R}_{+}\right): D \phi(\cdot, t) \in \mathcal{T}(S) \quad \forall t \in [a,b]\right\}
$$ 
\begin{definition}[free boundary integral Brakke flow, \cite{edelen18}]\label{def:fbbrakke}
We say a collection $\{\mu_t\}_{t \in I}$ of Radon measures is an $n$-dimensional free boundary integral Brakke flow in $\Omega$ if the following holds:
\begin{enumerate}[(1)]
\item for a.e. $t \in I$, $\operatorname{spt}(\mu_t)\subset \bar{\Omega}$ and $\mu_t=\mu_{V_t}$ for some integral $n$-varifold $V_t$ with free boundary in $S =\partial \Omega$, having
$$
S^T\left(H_{V_t}\right) \in L_{l o c}^2\left(\mathbb{R}^{n+1}, \mathbb{R}^{n+1} ; \mu_{V_t}\right);
$$
\item for any finite interval $[a,b] \subset I$, and every $\phi \in \mathcal{B T}(S, [a,b])$, we have 
\begin{multline}\label{equ:ede12}
\int \phi(\cdot, b) d \mu_b-\int \phi(\cdot, a) d \mu_a \\ \leq \int_a^b \int \left(-\left|S^T(H)\right|^2 \phi+S^T(H) \cdot D \phi+\partial_t \phi\right) d \mu_t d t.
\end{multline}
\end{enumerate}
\end{definition} 
It is easy to see that any classical mean curvature flow with free boundary $\left\{M_t^n\right\}$ is a free boundary Brakke flow, by taking $\mu_t=\mathcal{H}^n\left\llcorner M_t\right.$.  

\subsection{Monotonicity formula}
Inherited from Edelen's fundamental work \cite{edelen18}, we will keep on using his upper-semicontinuous monotone quantity. 
Recall from \cite[Def 5.0.1]{edelen18} that for $x\in \RR^{n+1}$ and $t<0$ one considers the adjoint heat kernel 

$$\rho(x,t)=(4\pi(-t))^{-\frac n2}e^{-\frac{|x|^2}{4(-t)}}$$
and the cut-off function 
$$
\phi_{\kappa}(x, t) =\left(1-\kappa^{-2}\left(\kappa^2 / (-t)\right)^{3 / 4}\left(|x|^2-\alpha (-t)\right)\right)_{+}^4.
$$
Here $\kappa$ is the cut-off radius chosen small depending on $S$ and $\alpha$ is some large number as determined in \cite[Thm 4.9]{edelen18}.
Moreover, recall that
$\tilde{x}=2 \zeta(x)-x$ denotes the reflection point of $x$ across $S$, where $\zeta(x)$ is the nearest point projection of $x$ onto $S$.
The reflected, truncated heat kernel centered at $X_0=(x_0,t_0)$ is defined to be
$$f_{S,\kappa,X_0}(x,t)=\phi_{\kappa}(x-x_0,t-t_0)\rho(x-x_0,t-t_0)+\phi_{\kappa}(\tilde{x}-x_0,t-t_0)\rho(\tilde{x}-x_0,t-t_0)$$

For $\mathcal{M} =\{\mu_t\}_{t \geq 0}$ a free boundary Brakke flow in $\Omega$, the reflected Gaussian area of $\mathcal{M}$ at $X_0=\left(x_0, t_0\right)$ is defined to be
$$
\Theta_\text{refl,$\kappa$}\left(\mathcal{M}, X_0, r\right)=\int f_{S, \kappa, X_0}\left(x, t_0-r^2\right) d \mu_{t_0-r^2}(x),
$$
where $r$ is bounded above by a constant depending on $\kappa$ and $t_0$. When $x_0$ is in the interior away from $S$, it is just the usual truncated Gaussian area for Brakke flow. One can also check that $\Theta_\text{refl,$\kappa$}$ is parabolic scaling invariant.
\begin{theorem}[monotonicity formula, \text{\cite[Thm 5.5]{edelen18}}]\label{thm:ede5.5}
Let $\mathcal{M} =\{\mu_t\}_{t \geq 0}$ be a free boundary Brakke flow in $\Omega$. For any $X_0=(x_0,t_0)\in \bar{\Omega}\times (0,\infty)$ and any $\kappa$ small enough, we have that
$$
r \mapsto e^{A \sqrt{r}} \Theta_{\mathrm{refl},\kappa}\left(\mathcal{M}, X_0, r\right)+ M r^2
$$
is increasing when $r$ is small. Here $A$, $M$ are constants depending only on $X_0$ and $\kappa$.
\end{theorem}
As corollaries of the theorem, one can define the Gaussian density by taking $r\to 0$. Since \cite[Cor 6.5]{edelen18} that this limit is independent of $\kappa$, one gets:
\begin{definition}[reflected Gaussian density]\label{def:ede6.5.1}
The reflected Gaussian density at a point is
$$
\Theta_\mathrm{refl}(\mathcal{M}, X)=\lim _{r \rightarrow 0} \Theta_{\mathrm{refl},\kappa}(\mathcal{M}, X, r).
$$
\end{definition}
Moreover, we have 
\begin{corollary}[boundedness of area ratios]\label{cor:boundgauss}
Let $\MM=\{\mu_t\}_{t\ge 0}$ be a free boundary Brakke flow in $\Omega$ and assume $\Theta_\mathrm{refl}(\MM,(p,T))<\infty$. Then
there exists some $\varepsilon>0$, s.t. $\dfrac{|M_{T-r^2}\cap B_r(p)|}{r^n}\le C$ for any $r<\varepsilon$.
\end{corollary}
Since the convergence of Brakke flows is defined with respect to compact exhaustions, this corollary is essential for running blow ups and applying Edelen's compactness theorem \cite[Thm 4.14]{edelen18}.

\subsection{Tangent flows}
Let $x_0 \in \bar{\Omega}, t_0>0$, and $\lambda_i \rightarrow \infty$. Write $X_0=\left(x_0, t_0\right)$. By \cite[Prop 6.2]{edelen18}, after passing to a subsequence, there is an ancient $\mathcal{M}^{\prime}$ so that
$$
\mathcal{D}_{\lambda_i}\left(\mathcal{M}-X_0\right) \rightarrow \mathcal{M}^{\prime}
$$
Here $\mathcal{M}^{\prime}$ is either a Brakke flow in $\mathbb{R}^{n+1}$ (if $x_0 \notin S$), or a free boundary Brakke flow in some half-space in $\mathbb{R}^{n+1}$ (if $x_0 \in S$).

If $\mathcal{M}^{\prime}$ has free boundary it can be reflected to obtain a Brakke flow $\tilde{\mathcal{M}}^{\prime}$ in $\mathbb{R}^{n+1}$ (without boundary). Otherwise simply let $\tilde{\mathcal{M}}^{\prime}=\mathcal{M}^{\prime}$. Then by \cite[Thm 6.4]{edelen18}, we know that 
\begin{theorem}[tangent flows]\label{thm:ede6.4}
The (reflected) Brakke flow $\tilde{\mathcal{M}}^{\prime}$  is an ancient selfshrinker, having density
$$
\Theta\left(\tilde{\mathcal{M}}^{\prime}\right)=\Theta_\mathrm{refl}\left(\mathcal{M}, X_0\right).
$$
\end{theorem} 

\section{Unit-regular and cyclic free boundary flows}
In this section, we will prove the preservation of some essential properties of free boundary Brakke flows under limits and reflection. Recall from the introduction that we have defined unit-regularity and cyclicity in Definition \ref{def:regcyc}.
The preservation under weak limits has been stated in Theorem \ref{thm:limitregcyc} and we put it in an equivalent form for convenience.
\begin{theorem*}[preservation under weak limits]\label{thm:limitregcyc1}
Suppose $\mathcal{M}^i$ are integral, unit-regular, cyclic free boundary Brakke flows in $\Omega_i$ that converge weakly to some free boundary Brakke flow $\mathcal{M}$ in $\Omega$, then $\mathcal{M}$ is also integral, unit-regular and cyclic. 
\end{theorem*} 

\begin{proof}
First of all, the preservation of integrality has already been shown in \cite[Thm 4.14]{edelen18}.

To prove the preservation of unit-regularity, we only need to show that if $X$ is a point of $\mathcal{M}$ at which the reflected Gaussian density is 1, then the $\mathcal{M}^i$ converge smoothly to $\mathcal{M}$ in a space-time neighbourhood of $X$.

Let $X=(x_0,t_0)\in\mathcal{M}$. If $x_0\in \Omega$, we can restrict to a small neighborhood to get Brakke flow without boundary and directly apply \cite[Thm 4.2]{sw20}. 

To deal with the case when $x_0 \in \partial \Omega$, we generalize the argument of \cite[Thm 4.2]{sw20} to the free boundary setting.
Suppose that $X_i \in \mathcal{M}^i$ converge to $X$ and $\Theta_\mathrm{refl}(\mathcal{M}, X)=1$. Let $\Theta_i=\Theta_\mathrm{refl}\left(\mathcal{M}^i, X_i\right)$. By upper semicontinuity of the reflected Gaussian density (see [Edelen Lem 7.2]), $\Theta_i \rightarrow 1$. Let $\mathcal{T}^i$ be a tangent flow to $\mathcal{M}^i$ at $X_i$. Then $\Theta_\mathrm{refl}\left(\mathcal{T}^i, O\right)=\Theta_i \rightarrow 1$, so (by [Edelen, Cor 8.4]), $\Theta_i=\Theta_\mathrm{refl}\left(\mathcal{T}^i, O\right)=1$ for all sufficiently large i. Since $\mathcal{M}^i$ is unit-regular, this implies that $\mathcal{M}^i$ is regular at $X_i$. Therefore, there is a space-time neighbourhood $\mathcal{U}$ of $X$ such that $\mathcal{M}^i \cap \mathcal{U}$ is smooth for all sufficiently large $i$. By choosing $\mathcal{U}$ small (and by monotonicity), we can arrange that the reflected Gaussian density ratios of $\mathcal{M}^i \cap \mathcal{U}$ are $\leq 1+\varepsilon$ (for any specified $\varepsilon>0$). The local regularity theory in [Edelen, Thm 8.1] then gives uniform $C^{2, \alpha}$ estimates (thus higher order by Schauder estimates) on the $\mathcal{M}^i \cap \mathcal{U}$.

Secondly, we prove that being cyclic is preserved.
By definition, we have $\operatorname{spt}(\partial [V^i_{\mu_t}])\subset S$. Note that $$\mathcal{M}^i\to \mathcal{M} \quad\text{as Brakke flows with boundary}$$ implies
$$ \mathcal{M}^i|_{\Omega}\to \mathcal{M}|_{\Omega} \quad\text{as Brakke flows without boundary}.$$
Thus, we can directly apply [White09, Thm 4.2] to infer that $\operatorname{spt}(\partial [V_t])\cap \Omega=\emptyset$, implying
$$\operatorname{spt}(\partial [V_t])\subset S.$$
This concludes the proof of the theorem.
\end{proof}

To consider the properties preserved under reflection, we start with the following elementary lemma.
\begin{lemma}[reflection of smooth flows]\label{lem:reflectsmooth}
If $\mathcal{M}$ is a smooth free boundary mean curvature flow in the halfspace $\mathbb{R}^{n+1}_+$, then the reflected flow is a smooth mean curvature flow in $\mathbb{R}^{n+1}$. 
\end{lemma}

\begin{proof}[Proof of Lemma \ref{lem:reflectsmooth}]
By definition, we only need to show that the reflected flow satisfies the PDE of MCF. The points away from the hyperplane $S=\partial \mathbb{R}^{n+1}_+$ automatically satisfy the equation by reflection.  We thus only consider the spacetime point $X_0=(x_0,t_0)$, where $x_0\in S$. Since the flow is orthogonal to $S$, the MCF is graphical in some mirror symmetric neighborhood $\UU$ of $X_0$. Therefore, 
we can locally express the MCF by a graph $(x,u(x))$ satisfying the following PDE:
\begin{equation}\label{equ:graphmcf}
\begin{cases}
\partial_t u=\sqrt{1+|D u|^2}\operatorname{div}\left( \dfrac{D u}{\sqrt{1+|D u|^2}} \right),
\quad &\text{in }\{x_1\ge 0\}\cap \UU, \\
u_{x_1}=0, \quad &\text{on }\{x_1=0\} \cap \UU.
\end{cases} 
\end{equation}
And we define the reflected solution 
\begin{equation}\label{equ:reflectu}
\tilde{u}(x)=\begin{cases}u(x),\quad x_1\ge 0 \\ u(\tilde{x}),\quad x_1< 0 \end{cases} 
\end{equation}
where $\tilde{x}$ is the reflection on the first coordinate.
To investigate the a priori regularity, we first observe that $\tilde{u}$ depends smoothly on $t$, since there is no transform on the time. Moreover, $\tilde{u}_{k}\ (2\le k\le n+1)$ are smooth away from the boundary $S:=\{x_1=0\}$. And $\tilde{u}$ must entirely be $C^1$ since $\tilde{u}_{x_1}=0$ on the boundary. 
We then aim to show that the second partial derivatives $\tilde{u}_{x_ix_j}$ exist and are continuous.
For $2\le i\le j\le n+1$, we have
$$
\lim_{x_1\to0^{+}}u_{x_ix_j}(x)=u_{x_ix_j}(0,x_2,\cdots)=\lim_{x_1\to0^{-}}u_{x_ix_j}(\tilde{x}).
$$
For $i=1$, $j\ge 2$, we always have $\tilde{u}_{x_1x_j}\equiv 0$. So, we only need to check for $\tilde{u}_{x_1x_1}$. 
Using the definition of $\tilde{u}$, we see that $s\mapsto \tilde{u}_{x_1}(s,x_2,\cdots,x_{n+1})$ is an odd function. Thus, $$\tilde{u}_{x_1x_1}(0^{-})=\lim_{s\to 0^{+}}\frac{\tilde{u}_{x_1}(-s,x_2,\cdots, x_{n+1})}{-s}=\tilde{u}_{x_1x_1}(0^+),$$
which implies $\tilde{u}\in C^2$. Thereby, the continuity of the mean curvature $H$ and the normal velocity $(\partial_t x)^{\perp}$ guarantees the MCF equation satisfied on $S$ and thus on the whole space. So, $\tilde{u}$ is a classical solution of (\ref{equ:graphmcf}) and the coefficients are $C^1$. Finally, applying Schauder estimates, we obtain the smoothness of $\tilde{u}$. 
\end{proof}

\begin{remark}\label{rmk:localref}
From the proof, we actually obtain a local property that if $\mathcal{M}$ is a smooth free boundary mean curvature flow in $\UU$ of a hyperplanar boundary $S$, then the reflected flow is a smooth mean curvature flow in $\tilde{\UU}$.
\end{remark}


Now let us prove that unit-regular, cyclic free boundary integral Brakke flows are preserved under reflection, which has already been stated in Theorem \ref{thm:reflectprop}. 
We state it again for convenience of the readers.
\begin{theorem*}[preservation under reflection]
Given an integral, unit-regular, cyclic free boundary Brakke flow in the halfspace $\mathbb{R}^{n+1}_+$, the reflected flow is also an integral, unit-regular, cyclic flow (without boundary) in $\mathbb{R}^{n+1}$.
\end{theorem*}
\begin{proof}
We denote the reflection across the hyperplane $S:=\partial \mathbb{R}^{n+1}_+$ by $A: \mathbb{R}^{n+1} \rightarrow \mathbb{R}^{n+1}$. Then the doubled varifold is expressed as $\tilde{V}=V+A_{\sharp} V$, where the pushforward of a varifold is defined as pushforward of the measure.
By \cite[Prop 3.1]{edelen18},  the doubled varifold is an integral varifold without boundary while by \cite[Prop 4.6]{edelen18}, the doubled flow is an integral Brakke flow without boundary. 

Noting that unit-regularity is a local property, 
we only need to consider points $X_0=(x_0,t_0)$ with $x_0\in S$. Because
$$\Theta_\mathrm{refl}(\mathcal{M},X_0)=\Theta(\tilde{\mathcal{M}},X_0)$$
for any such $X_0$ with $\Theta(\tilde{\mathcal{M}},X_0)=1$, there will be (in the original  flow setting) a small open neighborhood such that $t^{\prime} \mapsto \mathcal{M}_{t'} \cap B(x_0, \varepsilon)$ is a smooth free boundary mean curvature flow for $t^{\prime} \in\left[t_0-\varepsilon^2, t_0+\varepsilon^2\right]$. 
Then applying Lemma \ref{lem:reflectsmooth} and Remark \ref{rmk:localref}, we obtain the smooth mean curvature flow without boundary for $t^{\prime} \in\left[t_0-\varepsilon^2, t_0+\varepsilon^2\right]$, implying unit-regularity.\\

Finally, since cyclicity is an elliptic property, we only need to prove preservation under reflection for varifolds, which will automatically imply the flow case. 

Recall from \cite{whi09currents}, for any integral varifold $V$, the associated mod $2$ flat chain $[V]$ is the unique rectifiable mod $2$ flat chain, such that
$$\Theta([V],x)\equiv \Theta(V,x)\mod 2\ \text{ for }\mathcal{H}^n\text{ a.e. }x.$$
Since \cite[Sec 2.3]{whi09currents} implies that $A_{\#}$ commutes with $[\cdot]$, we have 
$$[\tilde{V}]=[V]+[A_{\#}V]=[V]+A_{\#}[V].$$
Taking $\partial$ on either side and we want to show the commutativity of $\partial$ with $A_{\#}$. 
To this end, recall that by the theory from \cite{fleming66}, any mod $2$ flat chain is the limit of some polyhedral chains, namely, $Q=\lim_{i\to \infty}P_i,$ where the limit is taken in the flat metirc. And the boundary of flat chains can be defined via $\partial Q:=\lim\limits_{i\to\infty}\partial P_i$, which is well-defined by the estimates of flat metric norm as in \cite[Sec 3]{fleming66}.
Moreover, by [Fle66, Sec 5], for any Lipschitz map $f$ one has $f_{\#}Q=\lim_{i\to\infty}f_{\#}P_i$. 
To show
$\partial A_{\#}=A_{\#}\partial$, it thus suffices to check for polyhedral chains that 
$$\partial A_{\#}P_i= A_{\#}\partial P_i,$$
which indeed holds by the linearity of $A$ in our special setting.
Moreover, by the continuity of $A_{\#}$, which is induced from a linear map, we get
\begin{equation}\label{equ:2chainonp}
\partial[\tilde{V}]=\partial[V]+A_{\#}\partial[V].
\end{equation}
Since $V$ is cyclic by assumption, we have $\spt{\partial [V]}\subset S$. Note that $A|_{S}=Id$ as a Lipschitz map, and thus $A_{\#}|_{S}=Id$ as mod 2 flat chain map by the perspective of polyhedral chains.
Thereby, we finally obtain 
$$\partial[\tilde{V}]=2\partial[V]=0.$$
This completes the proof.
\end{proof}

\begin{remark}
In fact the proof above shows that each property is inherited from the property itself instead of a combination of them.
\end{remark}

\section{Avoidance principle for free boundary Brakke flows}\label{sec:avoidance}
In this section, we are going to prove Theorem \ref{thm:0elliptic10.5} that the support of any free boundary Brakke flow satisfies the free boundary avoidance principle.

\subsection{Extending the class of test functions}
In this subsection, we are going to extend the test function class of Definition \ref{def:fbbrakke}.
We start with the following approximation lemma.
\begin{lemma}[existence of approximation]\label{lem:approxphi}
Given a nonnegative function $\phi\in C_c^1(\RR^{n+1}\times [a,b])$, there exists a nonnegative function $\phi_{\varepsilon}\in \mathcal{B T}(S, [a,b])$ such that 
\begin{equation}\label{equ:approxphi}
\begin{split}
|\phi-\phi_{\varepsilon}|& \le C\varepsilon,\\
|D\phi-D\phi_{\varepsilon}|&\le C,\\
|\partial_t\phi-\partial_t\phi_{\varepsilon}|&\le f_{\phi}(\varepsilon).
\end{split}
\end{equation}
The constants $C$ does not depend on $\varepsilon$, and $f_{\phi}(\varepsilon)\to 0$ when $\varepsilon\to 0$.
\end{lemma}
\begin{proof}
We first fix a compact set $K$ such that $\spt{\phi(\cdot,t)}\subset K$ for all $t$.
Take a two-sided tubular neighborhood $$U_{\varepsilon}=\{x\in \RR^{n+1}|\dist{x,S\cap K}<\varepsilon\text{ and is achieved only at }S\cap K\}$$
of part of the boundary $S$, where $\varepsilon$ is taken so small that there is a well-defined Fermi coordinates $x=(r,y)$
 such that the signed distance from $x$ to $S$ equals $r$ and is achieved only by $y\in S$. Recall the standard mollifier
$$\eta(x)=\begin{cases}e\exp\left(\dfrac{1}{\left|x\right|^2-1}\right),&|x|<1,\\0,&|x|\geq 1.\end{cases}$$
Construct the interpolating approximation 
\begin{equation}\label{equ:phiepsilon}
\phi_{\varepsilon}(r,y)=
\begin{cases}
\phi(r,y) (1-\eta(\frac{2r-\varepsilon}{\varepsilon}))+\eta(\frac{2r-\varepsilon}{\varepsilon})\phi(\frac{\varepsilon}{2},y),\quad \frac{\varepsilon}{2}\le r\le \varepsilon, \\
\phi(\frac{\varepsilon}{2},y),\quad -\frac{\varepsilon}{2} \le r\le \frac{\varepsilon}{2},\\
\phi(r,y) (1-\eta(\frac{2r+\varepsilon}{\varepsilon}))+\eta(\frac{2r+\varepsilon}{\varepsilon})\phi(\frac{\varepsilon}{2},y),\quad -\varepsilon\le r\le -\frac{\varepsilon}{2}, \\
\end{cases}
\end{equation}
and simply define $\phi_{\varepsilon}(x)=\phi(x)$ for $x\in \Omega\backslash U_{\varepsilon}$. Here and later we omit the variable $t$ for ease of notation. Observing when $|r|\ge \varepsilon$, then $\phi_{\varepsilon}(r,y)=\phi(r,y)$ and checking the coincidence up to first derivative at points of $|r|=\frac{\varepsilon}{2}$,
we have $C_c^1(\RR^{n+1}\times [a,b])$. Observing also that $D_{\nu_S}\phi_{\varepsilon}=0$ by construction, we thus get $\phi_{\varepsilon}\in \mathcal{B T}(S, [a,b])$. Note also that $\phi_{\varepsilon}\ge 0$ by construction.

 Since $\phi_{\varepsilon}$ and $\phi$ only differ in $U_{\varepsilon}$, we only need to check (\ref{equ:approxphi}) in $U_{\varepsilon}$. 
Since $\phi\in C^1$ and $S\cap K$ is compact,  we have 
\begin{equation}\label{equ:approxcondi1}
\sup_{U_{\varepsilon}\times [a,b]}|D\phi| \le C_1
\end{equation}
for some constant $C_1$ depending only on $\phi$. Moreover,
\begin{equation}\label{equ:approxcondi2}
\sup_{\frac{\varepsilon}{2}\le \pm r\le \varepsilon}\left|\partial_r \eta \left(\frac{2r\mp \varepsilon}{\varepsilon}\right)\right| \le \frac{C_2}{\varepsilon},
\end{equation}
for some constant $C_2$ depending only on $\phi$. Calculating via formula (\ref{equ:phiepsilon}), one can check that
\begin{equation}\label{equ:approxcheck1}
\sup_{{U_\varepsilon}\times[a,b]}\left|\phi(r,y)-\phi_{\varepsilon}(r,y)\right|\le C_1\cdot \varepsilon,
\end{equation}
and 
\begin{equation}\label{equ:approxcheck2}
\sup_{{U_\varepsilon}\times[a,b]}\left|\partial_r\phi(r,y)-\partial_r \phi_{\varepsilon}(r,y) \right|
\le C_1+C_1\cdot C_2,
\end{equation}
and 
\begin{equation}\label{equ:approxcheck2b}
\sup_{{U_\varepsilon}\times[a,b]}\left|\partial_{y_i}\phi(r,y)-\partial_{y_i} \phi_{\varepsilon}(r,y) \right|
\le 2C_1.
\end{equation}
Indeed, to verify (\ref{equ:approxcheck2}), we distinguish two cases.
If $|r|\le \frac{\varepsilon}{2}$, then $\partial_r \phi_{\varepsilon}=0$, and thus
$$\left|\partial_r\phi(r,y)-\partial_r \phi_{\varepsilon}(r,y) \right| =|\partial_r \phi(r,y)|
\le C_1,$$ 
and if $\frac{\varepsilon}{2}\le |r|\le \varepsilon$, then
$$\begin{aligned}
\left|\partial_r\phi(r,y)-\partial_r \phi_{\varepsilon}(r,y) \right| 
=& \left|\eta \partial_r\phi(r,y)+\partial_r \eta\left(\frac{2r\mp\varepsilon}{\varepsilon}\right) \left(\phi(r,y) -\phi(\frac{\varepsilon}{2},y)\right)\right| \\ 
\le& C_1+\frac{C_2}{\varepsilon}\cdot C_1\varepsilon.
\end{aligned}$$
Similarly, to verify (\ref{equ:approxcheck2b}), 
if $|r|\le \frac{\varepsilon}{2}$, then by (\ref{equ:approxcondi1}) we have
$$\left|\partial_{y_i}\phi(r,y)-\partial_{y_i}\phi_{\varepsilon}(r,y) \right| 
\le 2\sup_{U_{\varepsilon}\times [a,b]}|D\phi|\le 2C_1,$$ 
and if $\frac{\varepsilon}{2}\le |r|\le \varepsilon$, then
$$\left|\partial_{y_i}\phi(r,y)-\partial_{y_i}\phi_{\varepsilon}(r,y) \right| = \left|\eta \left(\phi_{y_i}(r,y) -\phi_{y_i}(\frac{\varepsilon}{2},y)\right)\right| 
\le 2C_1.
$$
Now (\ref{equ:approxcheck1}), (\ref{equ:approxcheck2}) and (\ref{equ:approxcheck2b}) together complete the first two claims of (\ref{equ:approxphi}) in the lemma.
Finally, to verify the last claim of (\ref{equ:approxphi}), by compactness of $\overline{U_{\varepsilon}}\times [a,b]$, the function $\partial_t \phi(r,y)$ is uniformly continous, hence
$$
\left|\partial_t\phi(r,y)-\partial_t\phi(\frac{\varepsilon}{2},y)\right|\le f_{\phi}(\varepsilon),\quad |r| \le \frac{\varepsilon}{2},
$$
where $f_{\phi}(\varepsilon)\to 0,$ when $\varepsilon\to 0$. And this also, by similar computation, leads to 
$$\begin{aligned}
\left|\partial_t\phi(r,y)-\partial_t \phi_{\varepsilon}(r,y) \right|
\le f_{\phi}(\varepsilon),\quad \frac{\varepsilon}{2} \le |r|\le \varepsilon,
\end{aligned}$$
which completes the proof of the lemma.
\end{proof}
Now we apply Lemma \ref{lem:approxphi} to prove the following proposition.
\begin{proposition}[extended test function class]\label{prop:c1class}
Given $\{\mu_t\}_{t \in I}$ a free boundary integral Brakke flow in $\Omega$, for any nonnegative $\phi\in C_c^1(\RR^{n+1}\times [a,b])$, we have 
\begin{multline*}
\int \phi(\cdot, b) d \mu_b-\int \phi(\cdot, a) d \mu_a \\ \leq \int_a^b \int \left(-\left|S^T(H)\right|^2 \phi+S^T(H) \cdot D \phi+\partial_t \phi \right) d \mu_t d t.
\end{multline*}
\end{proposition}
\begin{proof}
By Lemma \ref{lem:approxphi}, there is a nonnegative function $\phi_{\varepsilon}\in \mathcal{B T}(S, [a,b])$ such that the estimates (\ref{equ:approxphi}) hold. Then by Definition \ref{def:fbbrakke} of free boundary Brakke flow, we have 
\begin{equation}\label{equ:epsiloninequ}
\left.\int\phi_{\varepsilon} d\mu_t\right|_{t=b}-\left.\int\phi_{\varepsilon} d\mu_t\right|_{t=a}\le \int\int\left(-\left|S^T(H)\right|^2 \phi_{\varepsilon}+S^T(H)\cdot D\phi_{\varepsilon}+\partial_t\phi_{\varepsilon}\right).
\end{equation}
By \cite[Prop 4.13]{edelen18}, for our compact region $K$ from above we have
$$
\sup_{t\in[a,b]}\mu_t(K)+
\int_a^b \int_K\left|S^T(H)\right|^2 d \mu_t d t \leq C (K).
$$
Hence, we can apply Lemma \ref{lem:approxphi} to obtain the estimates
$$\left|\int\phi_{\varepsilon}d\mu_{t_0}-\int\phi d\mu_{t_0}\right|\le C\varepsilon,$$
and
$$\left|\int\int \left|S^T(H)\right|^2 \phi_{\varepsilon}d\mu_t dt-\int\int \left|S^T(H)\right|^2\phi d\mu_t dt\right|\le C\varepsilon.$$
Next, recall from the proof of Lemma \ref{lem:approxphi} that we have 
$|D\phi-D\phi_{\varepsilon}|\le C$ and $D\phi=D\phi_{\varepsilon}$ outside $U_{\varepsilon}$. This yields
\begin{multline*}
 \left|\int\int S^T(H)\cdot D\phi_{\varepsilon}d\mu_t dt-\int\int S^T(H) \cdot D\phi d\mu_t dt\right|\\
\le C\int_a^b \int_{\RR^{n+1}}\chi_{U_{\varepsilon}}(x)\cdot |S^T(H)| d\mu_tdt.
\end{multline*}
Now by definition of $S^T$, we have $S^TH(X)=0$ for $\mu_t-$a.e. $x\in S$, so applying the dominated convergence theorem we infer that
$$\begin{aligned}
& \lim\limits_{\varepsilon\to 0}\int_a^b \int_{\RR^{n+1}}\chi_{U_{\varepsilon}}(x)\cdot |S^T(H)| d\mu_tdt =0.
\end{aligned}$$
Finally, applying Lemma \ref{lem:approxphi} again we see that
$$\left|\int\int \partial_t\phi_{\varepsilon}d\mu_t dt-\int\int \partial_t\phi d\mu_t dt\right|\le C f_{\phi}(\varepsilon).$$
Thus, we have shown that the differences of terms in (\ref{equ:epsiloninequ}) all tend to $0$ when letting $\varepsilon\to 0.$
Therefore, we can pass to the limit in (\ref{equ:epsiloninequ}) when $\varepsilon \to 0$ to obtain the assertion of Proposition \ref{prop:c1class}.
\end{proof}

\subsection{Proof of the avoidance principle}
Before going further the proof of Theorem \ref{thm:0elliptic10.5}, let us first see the convenience that convexity of $\Omega$ provides us.
\begin{lemma}[minimal distance reached orthogonally]\label{lem:convexinner}
Assume $N$ is a free boundary smooth surface in a convex region $\Omega$. Then for any $x\in \bar{\Omega}$, if $\dist{x,N}$ is achieved at some point $x_0$ (not necessarily unique) of $N$, then the straight segment
$\overline{x_0x}$ meets $N$ orthogonally and $\overline{x_0x}\subset \bar{\Omega}$.
\end{lemma}
\begin{proof}[Proof of Lemma \ref{lem:convexinner}]
When $x_0\in \Omega\cap N$ is in the interior, the conclusion holds because otherwise it contradicts the minimizing distance condition by a simple variation. 

Therefore, we only need to consider the case when $x_0\in \partial\Omega\cap N$. And we assume by contradiction that the intersection is not orthogonal. By convexity of $\bar{\Omega}$, we have the segment $\overline{x_0x}\subset \bar{\Omega}$. Moreover, $N$ meets $S$ orthogonally by free boundary condition of $N$, and thus we can always find a curve $\gamma_{x_0}\subset N$ starting at $x_0$ such that $\angle(\overline{x_0x},\gamma_{x_0})<\frac{\pi}{2}$ by our assumption. This again contradicts the minimizing distance condition by a simple variation.

The last statement $\overline{x_0x}\subset \bar{\Omega}$ simply follows from the convexity of $\bar{\Omega}$.
\end{proof}
We will also use the following lemma for technical reasons.
\begin{lemma}[ball avoidance]\label{lem:keepaway}
If $\{\mu_{t}\}$ is a free boundary Brakke flow in $\Omega$ and $\operatorname{spt}(\mu_{t_0})\cap B_{\rho}(x_0)=\emptyset$, then $\operatorname{spt}(\mu_t)$ remains disjoint from $B_{\rho(t)}(x_0)$ for $t_0\le t\le t_0+\frac{\rho^2}{2n}$, where $\rho(t)=\sqrt{\rho^2-2n(t-t_0)}$.
\end{lemma}
\begin{proof}[Proof of Lemma \ref{lem:keepaway}]
Consider the test function
$$\varphi(x,t)=\left(1-\frac{|x-x_0|^2+2n(t-t_0)}{\rho^2}\right)^3_+.$$
Since $\mu_{t_0}(\varphi)=0$ applying Proposition \ref{prop:c1class} we only need to show 
$$\int_{t_0}^{t_1}\int -\varphi H^2+D \varphi\cdot H +\partial_t\varphi \le 0,$$
for any $t_1\le t_0+\frac{\rho^2}{2n}$. Applying the integration by parts (\ref{equ:fvf}), we get
\begin{multline}\label{equ:differtform0}
 \int\int-\varphi H^2+D \varphi \cdot \vec{H}+\partial_t \varphi d \mu_t \\
= \int\int (\partial_t-\operatorname{div}_{V_t}D-H^2)\varphi d \mu_t+\int\int D \varphi\cdot n_{V_t} d\sigma_{V_t}.
\end{multline}
By \cite[Prop 3.2]{edelen18} and the free boundary condition, the measure $\sigma_{V_t}$ is supported on the boundary $S$, and $n_{V_t}=\nu_S$ at $\sigma_V$-a.e. $x$.
Using this, we will first show that the boundary term has the good sign.
For any boundary point $(x,t)$, the gradient
$$D\varphi=-6\left(1-\frac{|x-x_0|^2+2n(t-t_0)}{\rho^2}\right)^2_+\cdot\frac{x-x_0}{\rho^2}$$
points in the opposite direction as $\overrightarrow{x_0x}$. Together with the convexity of $\Omega$ this yields
\begin{equation}\label{equ:varphibdy}
D \varphi \cdot \nu_S\le 0.
\end{equation}
Moreover, a direct computation yields
\begin{equation}\label{equ:backheat}
(\partial_t-\operatorname{div}_{V_t}D)\varphi=-24\left(1-\frac{|x-x_0|^2+2n(t-t_0)}{\rho^2}\right)_+\cdot\frac{|(x-x_0)^{V_t}|^2}{\rho^4}\le 0.
\end{equation}
Combining (\ref{equ:varphibdy}) and (\ref{equ:backheat}) into (\ref{equ:differtform0}), we achieve the proof of the lemma.

\end{proof}
Now we come to prove Theorem \ref{thm:0elliptic10.5}, which is the generalization of \cite[Sec 10.5]{ilm94}, and we will adjust Ilmanen's proof to the free boundary setting. 
For convenience, we restate our theorem in an equivalent form as follows:
\begin{theorem*}[avoidance principle]
If $\Omega$ is a convex domain, then the support of free boundary Brakke flow in $\Omega$ is a free boundary subsolution.
\end{theorem*}  
\begin{proof}
 Let $\left\{\mu_t\right\}_{t \in\left[a, b\right]}$ be a free boundary Brakke flow in $\Omega$. Let $\left\{N_t\right\}_{t \in\left[a, b\right]}$ be a smooth compact free boundary mean curvature flow such that spt $\mu_{a} \cap N_{a}=\emptyset$. 
Define
$$d(x,t)=\operatorname{dist}\left(x, N_t\right).$$
Fix $\gamma>0$ so small that
\begin{equation}\label{equ:distaway}
\operatorname{dist}\left(\operatorname{spt} \mu_{a}, N_{a}\right)>\gamma,
\end{equation}
and such that $d(x, t)$ is smooth on the set
$$
\left\{(x, t) \bigg| 0<d(x, t)<\gamma, x \in \Omega, a \leq t \leq b\right\}.
$$
Inspired by \cite[Sec 10.5]{ilm94} and \cite[Thm 34]{hw23}, we define $\phi(x,t)=h(d(x,t))$, where $h(d)$ is a nonnegative $C^2$ function on $[0,+\infty)$ such that 
\begin{equation}\label{equ:phidef}
\begin{cases} 
h(d)=(\gamma-d)^3_+,\text{  for  }d\ge \frac{\gamma}{2}, \\ 
h\equiv \gamma^3,\text{  for  }0\le d\le \frac{\gamma}{4}, \\ 
h(d)>0,\text{ for }d<\gamma, \\
h(d)=0,\text{ for }d \ge \gamma,
\end{cases} 
\end{equation}
and we will denote $s:=(\gamma-d)_{+}$ for simplicity.
Since $d(x,t)$ is smooth in $B_{\gamma}(N_t)\backslash N_t$ and we set $h$ to be a constant in a smaller neighborhood of $N_t$, we have that $\phi=h(d)$ is a globally well-defined $C^2$ function.
Since $\phi>0$ on $N_t$, we only need to show $\mu_{t}(\phi)=0$ for $t \in\left[a, b\right]$ to achieve the proof.
Note that $\mu_{a}(\phi)=0$ automatically holds by (\ref{equ:distaway}).
Let $$t_*=\sup \left\{t \in\left[a, b\right]: \mu_{t^{\prime}}(\phi)=0 \text{ for all }t^{\prime}\in [a,t]\right\}.$$
 Assume by contradiction that $t_*<b$. Then, there exists some $\tau>0$, such that
\begin{equation}\label{equ:phiposi}
\left.\int \phi(\cdot, t) d \mu_t\right|_{t_*}^{t_*+\tau} >0.
\end{equation}
By the semidecreasing property \cite[Prop 4.13]{edelen18}, we have $\mu_{t_*}(\phi)=0$. Thus, $\dist{\operatorname{spt} \mu_{t_*}, N_{t_*}} \geq \gamma$, so $\operatorname{spt} \mu_{t_*}$ is disjoint from any ball $B_{\gamma}(x_0)$ centered at a point $x_0\in N_{t_*}$.
 So, by Lemma \ref{lem:keepaway}, without loss of generality, for the same $\tau>0$ we have
\begin{equation}\label{equ:niceaway}
\dist{\spt{\mu_t},N_t}\ge \frac{\gamma}{2}, \text { for } t\in [t_{\ast},t_{\ast}+\tau].
\end{equation}
Since $\phi$ is a globally defined nonnegative compactly supported $C^1$ function from our construction (\ref{equ:phidef}), we can insert $\phi$ into the Brakke inequality from Proposition \ref{prop:c1class} for $t \in\left[t_*, t_*+\tau\right]$ to get, by applying the first variation formula (\ref{equ:fvf}), that
\begin{equation}\label{equ:differtform}
\begin{aligned}
\left.\int \phi(\cdot, t) d \mu_t\right|_{t_*}^{t_*+\tau} & \leq \int\int-\phi H^2+D \phi \cdot \vec{H}+\partial_t \phi d \mu_t \\
&= \int\int (\partial_t-\operatorname{div}_{V_t}D-H^2)\phi d \mu_t+\int\int D \phi\cdot n_{V_t} d\sigma_{V_t},
\end{aligned}
\end{equation}
similarly as in (\ref{equ:differtform0}).
We will first show that this new boundary term has the good sign.
\begin{claim}\label{cla:bdyterm}
At any boundary point in the region where $\dist{\spt{\mu_t},N_t}\ge \frac{\gamma}{2}$, we have $D \phi\cdot \nu_S \le 0$.
\end{claim}
\begin{proof}[Proof of Claim \ref{cla:bdyterm}]
For any boundary point $(x,t)$ in the well-posed region,
assume the distance is achieved at point $x_0\in N_t$, then by Lemma \ref{lem:convexinner} we know that $D d$ points in the same direction as $\overrightarrow{x_0x}$. Together with the convexity of $\Omega$ this yields
$$D d \cdot \nu_S\ge 0.$$ 
Since $D \phi=-3 s^2 D d$, this implies the assertion.
\end{proof}
The following claim shows that the bulk term has the good sign as well.
\begin{claim}\label{cla:innerterm}
We have $(\partial_t-\operatorname{div}_{V_t}D)\phi\le 0$, provided $\gamma$ is small enough. 
\end{claim}
\begin{proof}[Proof of Claim \ref{cla:innerterm}]
The argument is similar as in \cite[Sec 10.5]{ilm94} and \cite[Thm 34]{hw23}, but for convenience of the reader we provide the details.
Decomposing
$$D\phi=\nabla^{V_t}\phi+(D\phi\cdot \nu_{V_t})\nu_{V_t},$$
where $\nu_{V_t}(x, t)$ is a unit normal to the varifold $V_t$,
we see that
$$\operatorname{div}_{V_t}D\phi=\Delta_{V_t}\phi-D\phi\cdot H.$$
Together with
$$\frac{d}{dt}\phi=\partial_t\phi+D\phi\cdot H,$$
we therefore get
$$\partial_t\phi-\operatorname{div}_{V_t}D\phi=\frac{d}{dt}\phi-\Delta_{V_t}\phi.$$
Recall from \cite[(3.3)]{ecker}, we have the following pointwise calculation
\begin{equation}\label{equ:eckerheat}
(\frac{d}{dt}-\Delta_{V_t})\phi=(\partial_t-\Delta_{\RR^{n+1}})\phi+D^2\phi(\nu_{V_t},\nu_{V_t}).
\end{equation}
Plugging in the expression of $\phi$ from (\ref{equ:phidef}), as we are computing on $\spt{\mu_t}$ thanks to (\ref{equ:niceaway}), we have
$D^2\phi=-3s^2D^2d+6sDd\otimes Dd$, and therefore we obtain 
$$\Delta \phi=-3s^2\Delta d+6s,$$
where we denote $\Delta=\Delta_{\RR^{n+1}}$ for simplicity, and
$$D^2\phi(\nu_{V_t},\nu_{V_t})=-3s^2 D^2d(\nu_{V_t},\nu_{V_t})+6s(Dd\cdot \nu_{V_t})^2.$$
Together with (\ref{equ:eckerheat}) and $\partial_t\phi=-3s^2\partial_t d$ this yields 
\begin{equation}\label{equ:operphi}
(\partial_t-\Delta_{V_t})\phi=-3 s^2\left(\partial_t d-\Delta d+D^2 d(\nu_{V_t}, \nu_{V_t})\right)-6 s\left(1-|\nu_{V_t} \cdot D d|^2\right).
\end{equation}
Now, by Lemma \ref{lem:convexinner}, we know the normal exponential map 
of $N_t\subset \bar{\Omega}$ is given as restriction of the normal exponential map of $N_t$ in $\RR^{n+1}$,
and therefore the local behavior of $d$ will be the same to the case without boundary.
Hence, from \cite[Thm 6.1]{es91} for $d$ small enough, we get
\begin{equation}\label{equ:heatopr}
\partial_t d-\Delta d \geq 0.
\end{equation}
Moreover, since $D^2 d$ is a finite dimensional quadratic form that vanishes in direction $D d$, as can be seen by differentiating $|D d|^2=1$, we have 
\begin{equation}\label{equ:quadwithnvanish}
|D^2 d(\nu_{V_t},\nu_{V_t})|\le C(1-(\nu_{V_t}\cdot D d)^2).
\end{equation}
Substituting (\ref{equ:heatopr}) and (\ref{equ:quadwithnvanish}) into (\ref{equ:operphi}), we conclude that
\begin{equation}\label{equ:operphi2}
(\partial_t-\Delta_{V_t})\phi \leq (3 C s^2 -6 s)(1-(\nu_{V_t}\cdot D d)^2) \le 0,
\end{equation}
provided $\gamma\ge s$ is small enough. This proves the claim.
\end{proof}
Finally, we conclude from (\ref{equ:differtform}) and the two claims that
$$
\left.\int \phi(\cdot, t) d \mu_t\right|_{t_*}^{t_*+\tau} \le 0,
$$
contradicting (\ref{equ:phiposi}). Hence $t_*=b$, and $\mu_t(\phi)=0$ for all $t \in\left[a, b\right]$, proving the theorem. 
\end{proof}

\section{Free boundary inner/outer flow and its matching Brakke flow}
In this section, we will generalize some results of outer and inner flow to our free boundary setting and construct matching free boundary Brakke flows for the outer and inner flow.
\subsection{Outer and inner free boundary flow}
Recall that in (\ref{equ:hwcpam2}) from the introduction we have defined the outer and inner flow in the free boundary setting. Generalizing the results from Appendix of \cite{hw20}, we are going to establish the relationship between $M_t$ and $\partial F_t(K)$ in the free boundary case. Before this, we will prove the following lemma as a technical ingredient.

\begin{lemma}[ball comparison]\label{lem:centerint}
For $x\in \partial\Omega$, and any $r>0$, denoting
\begin{equation}\label{equ:br}
B_r(x):=\{y\in\bar{\Omega}:|y-x|\le r\},
\end{equation}
there exists some $\tau>0$, such that $x\in\inte{F_{t}(B_r(x))}$ for $0<t\le \tau$.
\end{lemma}
\begin{proof}
We construct a smooth free boundary initial surface $N\subset B_r$ which encloses $x$ in the interior,
by describing in Fermi coordinates $(r,y)$ centered at $x$ that $N=\{r^2+|y|^2=\delta^2\}$, with $x=(0,0)$ and $\delta$ small enough. Since $N$ is smooth and encloses $x$, standard short-time existence of smooth free boundary mean curvature flow guarantees that $x$ remains enclosed by $N_t$ for all sufficiently small $t>0$, which in light of Proposition 2.3 implies the assertion.
\end{proof}
We are now ready to establish the following relationship.
\begin{proposition}[inner/outer flow]\label{prop:bdyflowrelation}
Let $F_t(K)$, $M_t$ be as described in the introduction and always denote the boundary operator to be $\partial(\cdot):=\cdot\backslash \inte{\cdot}$. Then, the free boundary level set flow and the outer flow are related as follows:
\begin{enumerate}[(i)]
\item $\partial F_t(K) \subseteq M_t,$
\item $M_t\subseteq F_t(M)$,
\item 
$
M_t=\lim _{\tau \uparrow t} \partial F_{\tau}(K).
$
\end{enumerate}
Properties for $M^{\prime}$ and $K^{\prime}$ are similar.
\end{proposition}
\begin{proof}
(i) simply follows from the definitions and the closeness of $\mathcal{K}$.

(ii) is the generalization of \cite[Prop A.3]{hw20}. We only need to check that $M_t$ is a free boundary subsolution by (\ref{equ:levelsetflow}) in the introduction. Namely, given any smooth compact free boundary mean curvature flow $\{N_t\}_{t\in[a,b]}$ in $\Omega$ with $N_a\cap M_a=\emptyset$, we want to show $N_t\cap M_t=\emptyset$ for all $t \in[a,b]$. We can assume that $N_t$ is connected. Then, by (i), we have either $N_a \subseteq K_a \backslash M_a$ or $N_a\cap F_a(K)=\emptyset$. The latter case will lead to $N_t\cap F_t(K)=\emptyset$ and will imply the result since $M_t\subseteq F_t(K)$. Thus it suffices to check the first case. Similarly as in \cite[Prop A.3]{hw20}\footnote{More precisely, we are following the argument from their preprint arXiv: 1704.00431v3.}, set $\mathcal{G}:=\operatorname{Int}_{\bar{\Omega}\times [a,b]}(\mathcal{K})$ and let $\mathcal{G}^*$ be the union of all space-time sweepouts of smooth free boundary flows
$\{\Sigma_t\}_{t\in [a,b]}$,
such that $\Sigma_a \times\{a\} \subseteq \mathcal{G}$. Then $\mathcal{G}^*\subseteq \mathcal{K}$ is relatively open in $\bar{\Omega} \times [a,b]$ and thus is disjoint from $\partial \mathcal{K}$. So, $N_t\cap M_t=\emptyset$ for all $t \in [a,b]$ as desired.

(iii) is the generalization of \cite[Thm B.2]{hw20}. By (i) and closeness of the space-time track of $M_t$,  we know for every $\varepsilon>0, \partial F_{\tau}(K)$ lies in the $\varepsilon$-neighborhood of $M_{t}$ for all $\tau<t$ sufficiently close to $t$. 
 Conversely, 
suppose by contradiction that there is a sequence $\tau_i \uparrow t$, a point $x \in M_t$, and an $r>0$ such that
\begin{equation}\label{equ:hwcpam24}
\operatorname{dist}\left(x, \partial F_{\tau_i}(K)\right)>r,\quad\text{for all $i$.}
\end{equation}
Fix a $\tau=\tau_i$ sufficiently close to $t$ that $x$ is still in the interior of $F_{t-\tau}(B_r(x))$ which is possible by Lemma \ref{lem:centerint}. Thus, $(x, t)$ is in the interior of the spacetime track $\mathcal{B}$ of $t \in[\tau, \infty) \mapsto F_{t-\tau}(B_r(x))$. Note that $B_r(x)\not\subset F_\tau(K)$, since otherwise we would have $\mathcal{B}\subset \mathcal{K}$, yielding $(x, t)\in\operatorname{Int}_{\bar{\Omega}\times [a,b]}(\mathcal{K})$, which is impossible since $(x, t) \in \partial \mathcal{K}$. Likewise, we also have $B_r(x)\cap F_\tau(K)\neq \emptyset$, since otherwise we would get $\mathcal{B}\cap\mathcal{K}= \emptyset$, which is impossible since $x \in F_{t-\tau}(B_r(x)) \cap M_t$ and $M_t\subset F_t(K)$ by (ii).
Hence, $B_r(x)$ must contain a point in $\partial F_\tau(K)$, contradicting (\ref{equ:hwcpam24}).

\end{proof}
Moreover, we have the following relationship between spacetime tracks of free boundary flows.
\begin{proposition}[intersection of outer and inner flow]\label{prop:intermeasure0}
Suppose $\Omega$ convex and $K\subset \bar{\Omega}$ is a closed set,  such that $\partial K=:M$ is a free boundary hypersurface in $\Omega$, where $\partial(\cdot):=\cdot\backslash \inte{\cdot}$ as usual. Then, the free boundary level set flow $\KK$ starting from $K$, the flow $\KK^{\prime}$ starting from $K^{\prime}=\overline{\bar{\Omega}\backslash K}$ and the flow $\MM$ starting from $M$, satisfy $\KK\cap \KK^{\prime} = \MM$.
\end{proposition}

\begin{proof}
We will generalize the proof of \cite[Thm 1.8]{bk23}.
We may choose a sequence of closed sets $K_j$ such that $K_{j+1}\subset \inte{K_j}$ and $\bigcap_j K_j=K$. 
Then, by Proposition \ref{prop:fbprop} the free boundary level set flows $\mathcal{K}_j$ starting from $K_j$ satisfy
$$
\mathcal{K}_1\supset \mathcal{K}_2 \supset \ldots, \quad \mathcal{K}=\bigcap_j \mathcal{K}_j .
$$
Similarly, we can choose a sequence of closed sets $K_j^{\prime}$, such that $K^{\prime}_{j+1}\subset \inte{K^{\prime}_j}$ and $\bigcap_j K^{\prime}_j=K^{\prime}$, and consequently
$$
\mathcal{K}^{\prime}_1\supset \mathcal{K}^{\prime}_2 \supset \ldots, \quad \mathcal{K}^{\prime}=\bigcap_j \mathcal{K}^{\prime}_j.
$$
In particular, we have 
$$
\mathcal{K}\cap \mathcal{K}^{\prime} = \bigcap_j \mathcal{K}_{j} \cap \mathcal{K}_{j}^{\prime}.
$$
Now by the avoidance principle and by Proposition \ref{prop:bdyflowrelation} we have $\partial \mathcal{K}_{j} \cap \mathcal{K}=\emptyset$ and $\partial \mathcal{K}_{j}^{\prime} \cap \mathcal{K}^{\prime}=\emptyset$,
where $\partial(\cdot):=\cdot\backslash \operatorname{Int}_{\bar{\Omega}\times [0,+\infty)}(\cdot)$.
 Combined with the inclusion property, this implies $\partial \mathcal{K}_{j} \subset \mathcal{K}^{\prime} \subset \operatorname{Int}_{\bar{\Omega}\times [0,+\infty)}(\mathcal{K}_{j}^{\prime})$ and $\partial \mathcal{K}_{j}^{\prime} \subset \mathcal{K} \subset \operatorname{Int}_{\bar{\Omega}\times [0,+\infty)}(\mathcal{K}_{j})$.
Consequently, any $X\in \KK_j\cup \KK_j^{\prime}$ is either in the interior of $K_j$ or in the interior of $K_j^{\prime}$. Since $\KK_j\cup \KK_j^{\prime}$ is also closed and nonempty, it follows that 
\begin{equation}\label{equ:unionwhole}
\mathcal{K}_{j} \cup \mathcal{K}_{j}^{\prime} =\bar{\Omega} \times[0, \infty).
\end{equation}
Moreover, it also follows that
\begin{equation}\label{equ:interbdy}
\partial\left(\mathcal{K}_{j} \cap \mathcal{K}_{j}^{\prime}\right)=\partial \mathcal{K}_{j} \dot{\cup} \partial \mathcal{K}_{j}^{\prime}
\end{equation}
by the key observation $\mathcal{K}_{j}\backslash \operatorname{Int}_{\bar{\Omega}\times [0,+\infty)}(\mathcal{K}_{j}) \subset \mathcal{K}_{j}\cap \mathcal{K}_{j}^{\prime}\backslash (\operatorname{Int}_{\bar{\Omega}\times [0,+\infty)}\left(\mathcal{K}_{j}\cap \mathcal{K}_{j}^{\prime})\right)$, which is true because of $\partial \mathcal{K}_{j} \subset \mathcal{K}^{\prime} \subset \operatorname{Int}_{\bar{\Omega}\times [0,+\infty)}(\mathcal{K}_{j}^{\prime})$.

Now we are going to show that $\mathcal{I}_j:=\mathcal{K}_{j} \cap \mathcal{K}_{j}^{\prime}$ is a free boundary subsolution.
Let $[a, b]\subset[0,\infty)$ be arbitrary, and let
$
\{\Sigma_t\}_{t\in[a,b]}
$
be a smooth free boundary mean curvature flow with $\Sigma_{a}$ connected such that
$$
\Sigma_{a} \cap \mathcal{I}_{j,a} = \emptyset.
$$
Using (\ref{equ:interbdy}), we see that $\Sigma_a$ is disjoint from either $\KK_{i,a}$ or $\KK_{j,a}^{\prime}$. Hence, $\Sigma_t$ is disjoint from either $\KK_{j,t}$ or $\KK_{j,t}^{\prime}$ for all $t\in[a,b]$.
This proves that $\mathcal{I}_j$ satisfies the free boundary avoidance principle. 
Since these flows form a decreasing sequence of subsets, we obtain that their intersection $\mathcal{K} \cap \mathcal{K}^{\prime}$ is also a free boundary subsolution, so it must be contained in $\mathcal{M}$. On the other hand, $\mathcal{M}$ is contained in both $\mathcal{K}$ and $\mathcal{K}^{\prime}$ by the inclusion property. This proves that
$
\mathcal{M}=\mathcal{K} \cap \mathcal{K}^{\prime},
$
as desired.
\end{proof}

\subsection{The matching free boundary Brakke flow}

In this subsection, we are going to prove Theorem \ref{thm:hwcpamB6} that we can match a nicely behaved free boundary Brakke flow for a given free boundary outer flow. Our proof generalizes the one from \cite[Thm B6]{hw20}, and we will include more details for clarification.
Before proving the theorem, we will provide the following lemma for technical reasons.
\begin{lemma}[free boundary perturbation]\label{lem:phivariation}
Assume $M^0$ is a smooth compact free boundary surface in a convex region $\Omega$. Then, there exists a smooth function $\phi: \bar{\Omega} \rightarrow \mathbb{R}$, such that $\phi^{-1}(0)=M^0$ and $(D\phi)|_{M^0}=\nu_{M^0}$ and the level sets $M^s:=\phi^{-1}(s)$ are also free boundary smooth surfaces in $\Omega$ for any $s\in (-\varepsilon,\varepsilon)$. Moreover, the associated Radon measures
$
\mu_{M^s} := \mathcal H^n\llcorner M^s
$
converge weakly to $\mu_{M^0}$ as $s\to 0$.
\end{lemma}
\begin{proof}
Since $S:=\partial\Omega$ is smooth and $M^0$ is compact, we can pick the neighborhood $D_1\subset\subset D_2 \subset S$ of $S\cap M^0$, and a $\delta>0$, such that the signed distance function $d_S:=d(\cdot,S)$ is smooth in the tubular neighborhood
$$
U_{\delta}=\left\{x\in \RR^{n+1}\bigg||d_S|<\delta\text{ and $d_S$ is achieved only at }D_2\right\}.
$$
Consider the vector field 
$$
\widetilde{\nu}_{S}:=D d_{S}\quad\text{in }U_\delta.
$$
Then, $\widetilde{\nu}_{S}=\nu_{S}$ on $S\cap U_{\delta}$.
Similarly, we can also extend $\nu_{M^0}$, to a smooth vector field $\widetilde{\nu}_{M^0}$ in a tubular neighborhood of $M^0$ in $\bar{\Omega}$. Moreover, 
for $x\in U_\delta$, define the smooth projection
$$
\pi(w)
:= w- \langle \widetilde{\nu}_{S}(x), w\rangle \widetilde{\nu}_{S}(x),
$$
onto the tangent space of the level set of $d_{S}$ through $x$.
Then, by the free boundary condition, we have
$$
\pi(\widetilde{\nu}_{M^0})=\nu_{M^0}\quad\text{on }\partial M^0.
$$
Choose a smooth cutoff function $\chi$ such that
$$
\chi\equiv1 \text{ near } D_1,
\qquad
\chi\equiv0 \text{ away from } U_\delta.
$$
Define
$$
V:=\chi\cdot\pi(\widetilde{\nu}_{M^0})+ (1-\chi)\widetilde{\nu}_{M^0}.
$$
Then, $V$ is smooth and satisfies $V=\nu_{M^0}$ on $M^0$ and is tangential to $S$ along $S$.
Now, let $\Phi_s$ be the flow generated by $V$.
Then, $\Phi_s$ preserves the boundary, $\Phi_s(\partial M^0)\subset S$, and the free boundary conditions. 
Moreover, $\Phi_s$ is a diffeomorphism near $M^0$ for sufficiently small $\varepsilon>0$. In conclusion, 
$$
M^s:=\Phi_s(M^0)\subset\overline{\Omega}
$$
is a smooth hypersurface with $\partial M^s\subset\partial\Omega$ satisfying the free
boundary condition.

Finally, we define $\phi$ in a neighborhood of $M^0$ by
$$
\phi(\Phi_s(x)):=s,\quad x\in M^0,\ |s|<\varepsilon.
$$
Then, $\phi$ is smooth with $D \phi\neq0$ on $M^s$, and
$$
\phi^{-1}(s)=M^s.
$$
Extend $\phi$ smoothly to all of $\overline{\Omega}$ such that the value of $(-s,s)$ will be never taken again and this completes the construction.

Finally, since by construction, we have $M^s \to M^0$ in $C^1$ as $s\to 0$, in particular we have
$
\mu_{M^s} := \mathcal H^n\llcorner M^s \rightharpoonup \mu_{M^0}
$ as $s\to 0$.
\end{proof}

Now, for convenience, we restate our theorem in an equivalent form as follows:
\begin{theorem*}[matching free boundary Brakke flows]
Given $\Omega$ convex and any $K\subseteq \bar{\Omega}$ with $M:=\partial K$ being a smooth compact free boundary surface, set $\mu=\mathcal{H}^n\llcorner\partial K$. Then, there exists a unit-regular, cyclic free boundary integral Brakke flow $\{\mu_t\}_{t\ge0}$, such that $\mu_0=\mu$ and the space-time support of the flow is the space-time set swept out by the outer flow $M_t$. The result for inner flow similarly holds by taking $K^{\prime}=\overline{\bar{\Omega}\backslash K}.$
\end{theorem*}
\begin{proof}
There exist compact regions $K_i$ with smooth boundaries, such that
\begin{enumerate}[(1)]
\item For each $i$, $K_{i+1}\subset \operatorname{Int}_{\bar{\Omega}}(K_i)$.
\item $\cap K_i=K$.
\item $\mathcal{H}^n\llcorner\partial K_i\rightharpoonup \mu.$
\item The free boundary level set flow of $\partial K_i$ never fattens.
\end{enumerate}
Indeed, let $M^s=\phi^{-1}(s)$ be a variation of $M$ as provided by Lemma \ref{lem:phivariation}. Then, considering $\partial K_i=M^{s_i}$ for any strictly decreasing sequence $s_i\to 0$ conditions (1)-(3) are satisfied. 
Moreover, since $\sum\limits_{s}\HH^n(M^s_t)<\infty$, for all but countably many $s$, we get $\operatorname{Int}(M^s_t)=\emptyset$ (first we only get this for all rational $t$, but using the inclusion property from Proposition \ref{prop:fbprop} this actually implies the statement for all $t$). Namely, at most countably many $M^s$ fatten under free boundary level set flow, so choosing $s_i$ suitably condition (4) also holds.

\medskip

We now set up the elliptic regularization following \cite{ilm94,edelen18}. Consider the functional 
\begin{equation}\label{equ:iepsilon}
I_\varepsilon(P)=\frac{1}{\varepsilon} \int e^{-z / \varepsilon} d \mu_P,
\end{equation}
on the space of integral $(n+1)$-currents in $\mathbb{R}^{n+2}$, where $z$ is the $\mathbb{R}$ component of $\mathbb{R}^{n+1} \times \mathbb{R}$, and $\mu_P$ is the mass measure associated to the current $P$, i.e. 
$$
\mu_P(U):=\sup_{|\omega|\le 1,\ \omega\in C_c^{\infty}(\Lambda^{n+1}U)} P(\omega).
$$

Let $P^{i,\varepsilon}$ be a minimizer of $I_{\varepsilon}$ in the class of integral $(n+1)$-currents with $\partial P=\partial K_i\times \{0\}$ and $\mu_P$ supported in $\bar{\Omega}\times [0,+\infty)$.
Consider the integral $(n+1)$-varifold $V^{i,\varepsilon}$ with $\mu_{V^{i,\varepsilon}}=\mu_{P^{i,\varepsilon}}$. Then, as observed in \cite[Prop 9.3]{edelen18} we have 
$$\delta V^{i,\varepsilon}(X)=\int \frac{\partial_z^{\perp}}{\varepsilon}\cdot X d\mu_{V^{i,\varepsilon}}$$
for any $X \in \mathcal{T}\left(\partial \Omega \times \RR \right) \cap C_c^1(\{z>0\})$. 
So, $V^{i,\varepsilon}$ has free boundary in $\partial \Omega\times (0,\infty)$.
Consequently, if we set 
\begin{equation}\label{equ:translatep}
P^{i,\varepsilon}_t:=(\sigma_{-t/\varepsilon})_{\#}P^{i,\varepsilon},
\end{equation}
where $\sigma_s(x,z)=(x,z+s)$, then $\mu_t:=\mu_{P^{i,\varepsilon}_t}$ for $t\in (\varepsilon^{\frac 12},\infty)$ is an $(n+1)$-dimensional integral Brakke flow in $\Omega\times (-\varepsilon^{-\frac 12},\infty)$ with free boundary in $\partial \Omega \times (-\varepsilon^{-\frac 12},\infty)$, c.f. \cite[Lem 8.7]{ilm94}.
 Note that it is unit-regular and cyclic as well. 
Indeed, we can consider the region enclosed by $P^{i,\varepsilon}_t$ and $\{z=-\frac{t}{\varepsilon}\}$, namely a closed set $E^{i,\varepsilon}_t\subset \bar{\Omega}\times \RR$ of locally finite perimeter, such that the induced current $[E^{i,\varepsilon}_t]$ satisfies 
\begin{equation}\label{equ:bydofe}
\partial [E^{i,\varepsilon}_t]\llcorner(\Omega \times \RR)=P^{i,\varepsilon}_t \cup (K_i\times \{-\frac{t}{\varepsilon}\}).
\end{equation}

\medskip

We aim to take the limit for the pair $(\mu^{i,\varepsilon}_t,E^{i,\varepsilon}_t)$ inspired by \cite[Sec 11]{ilm94}, and \cite[Sec 5]{whi-size}.
By \cite[Thm 9.5, Prop 9.6]{edelen18} and Theorem \ref{thm:limitregcyc}, after passing to a subsequence 
$\varepsilon_k\to 0$, the flows $\{\mu_t^{i,\varepsilon_k}\}_{t\in(\varepsilon_k^{\frac12},\infty)}$ converge to a unit-regular, cyclic free boundary integral Brakke flow $\{\tilde{\mu}^i_t\}_{t\in(0,\infty)}$ in $\Omega\times\RR$, which splits as $\tilde{\mu}^i_t=\mu^i_t\times \mathcal{H}^1$, where $\{\mu^i_t\}_{t\in (0,\infty)}$ is an $n$-dimensional unit-regular, cyclic free boundary integral Brakke flow in $\Omega$ starting from $\partial K_i$.
Moreover, by the compactness theorem for $\operatorname{BV}_{\operatorname{loc}}$ (\cite[Thm 5.5]{eg15}), after passing to another subsequence we can assume that the characteristic functions of $E_t^{i,\varepsilon_k}$ converge in $L^1_{\operatorname{loc}}$ to the characteristic function of some closed set $\tilde{E}^i_t$ with locally finite perimeter. As observed by Ilmanen \cite{ilm94}, since $E_{t+h \varepsilon}^{i,\varepsilon}=E_t^{i,\varepsilon} - h e_z$, the limit splits as $\tilde{E}^i_t = E_t^i\times\RR$, where $E_t^i\subset \RR^{n+1}$.
Hereby, we have defined the subsequential convergence of the pair
\begin{equation}\label{equ:pairconverge}
(\mu^{i,\varepsilon_k}_t,E^{i,\varepsilon_k}_t)\stackrel{\varepsilon_k\to0}{\longrightarrow}(\tilde{\mu}^i_t,\tilde{E}^i_t).
\end{equation}


\medskip

We first want to reveal some relation between $\tilde{\mu}_t^i$ and $\tilde{E}^i_t$.
\begin{claim}\label{cla:1}
For any $t>0$, we have
$$\partial^*_{\Omega\times \RR} \tilde{E}^i_t\subseteq \spt{\tilde{\mu}^i_t},$$ 
where
$$\partial^{*}_{\Omega\times \RR} E :=
\left\{
x \in \Omega\times \RR :
\lim_{r \to 0}
\frac{D\chi_E(B_r(x))}{|D\chi_E|(B_r(x))}
\ \text{exists and has unit norm}
\right\}.
$$
\end{claim}
\begin{proof}[Proof of Claim \ref{cla:1}]
For any $x \in \partial^*_{\Omega\times \RR}  \tilde{E}^i_t$, by our definition and by De Giorgi's theorem (c.f. \cite[Thm 15.5, Thm 15.9]{maggi12}), the tangent space of $\tilde{E}^i_t$ at $x$ is a half space.
Moreover, by the Gauss-Green formula (c.f. \cite[Thm 9.3]{maggi12}) the reduced boundary can be identified with the boundary of the induced current.
Then, by the local BV convergence of characteristic functions, there are $x_k \in \spt \partial [E^{i,\varepsilon_k}_t]=\spt{\mu^{i,\varepsilon_k}_t}$, such that
$$
x_k \rightarrow x \text {. }
$$
Since $\mu^{i,\varepsilon_k}_t$ converges to $\tilde{\mu}^i_t$ as Radon measure, $\spt{\mu^{i,\varepsilon_k}_t}$ converges  to $\spt{\tilde{\mu}^i_t}$ in the Hausdorff sense. Thus, we have $x\in \spt{\tilde{\mu}^i_t}$, which proves the claim. 
\end{proof}

\begin{claim}\label{cla:2}
For any $t>0$, we have $\HH^{n+1}(E^i_t \Delta F_t(K_i))=0$.
\end{claim}
\begin{proof}[Proof of Claim \ref{cla:2}]

First, we are going to show $\mathcal{H}^{n+1}(E^i_t \backslash F_t(K_i))=0$ by the avoidance principle.
Since the sets $E^{i,\varepsilon}_t$ are not uniquely defined due to null sets, we refine them as follows.
Define $G^{i,\varepsilon}_0=E^{i,\varepsilon}_0\backslash E^{{i,\varepsilon},(0)}_0$, where 
$$E^{{i,\varepsilon},(\theta)}_0=\{x\in E^{i,\varepsilon}_0|\Theta(E^{i,\varepsilon}_0,x)=\theta\},$$
and where $\Theta$ is the upper volume density. Then, by the Lebesgue density theorem we have $\HH^{n+2}(E^{{i,\varepsilon},(0)}_0)=0$, and by De Giorgi's theorem we have $\partial^* E^{i,\varepsilon}_0\subset E^{{i,\varepsilon},(\frac 12)}_0$. Moreover, by the locally finite perimeter assumption the essential boundary equals the closure of the reduced boundary, i.e. $$E^{i,\varepsilon}_0 \backslash (E^{{i,\varepsilon},(1)}_0 \cup E^{{i,\varepsilon},(0)}_0)=\overline{\partial^* E^{i,\varepsilon}_0},$$ which implies $G^{i,\varepsilon}_0 =\operatorname{Int}_{\bar{\Omega}\times \RR}{E^{i,\varepsilon}_0}\cup \overline{\partial^* (E^{i,\varepsilon}_0)}$. 
So, we refine $E^{i,\varepsilon}_t:=(\sigma_{-t/\varepsilon})_{\#}G^{i,\varepsilon}_0$, which still satisfies (\ref{equ:bydofe}). 
Together with the fact from \cite[Thm 4.4]{giu84} that for any closed set of finite perimeter the closure of the reduced boundary is equal to the boundary of the interior, we get \begin{equation}\label{equ:enicebdy}
E^{i,\varepsilon}_t =\operatorname{Int}_{\bar{\Omega}\times \RR}{E^{i,\varepsilon}_t}\cup \partial \operatorname{Int}_{\bar{\Omega}\times \RR}{E^{i,\varepsilon}_t}.
\end{equation}
Now, fix $\varepsilon>0$, and let $\{\Sigma_t\}_{t\in[t_0,t_1]}$ be a compact smooth free boundary mean curvature flow in ${\Omega}\times [-\frac{t_0}{\varepsilon}, \infty)$ with $\Sigma_{t_0}\cap E^{i,\varepsilon}_{t_0}=\emptyset$. In particular, we have $\Sigma_{t_0}\cap\  \spt{\mu_{P^{i,\varepsilon}_{t_0}}}=\emptyset$. Moreover, since $\mu_{P^{i,\varepsilon}_t}$ is a free boundary Brakke flow as explained in (\ref{equ:translatep}), the avoidance principle of free boundary Brakke flows (Theorem \ref{thm:0elliptic10.5}) gives $\Sigma_{t}\cap\  \spt{\mu_{P^{i,\varepsilon}_{t}}}=\emptyset$ for any $t\in [t_0,t_1]$. Furthermore, thanks to (\ref{equ:enicebdy}), we have in ${\Omega}\times [-\frac{t_0}{\varepsilon}, \infty)$ that $\partial (E^{i,\varepsilon}_{t} \cap \{z\ge -\frac{t_0}{\varepsilon}\}) \subset \spt{\mu_{P^{i,\varepsilon}_{t}}}$ for $t\ge t_0$. Thus, we must also have $\Sigma_{t}\cap E^{i,\varepsilon}_{t}=\emptyset$ for all $t\in [t_0,t_1]$, which shows that the refined $E^{i,\varepsilon}_{t}$ is a free boundary subsolution. Therefore, we have $$\HH^{n+2}(E^{i,\varepsilon}_t\backslash F_t(K_i\times \RR))=0.$$
Taking $\varepsilon\to 0$ and splitting the $\RR$-factor, we finally obtain
$$\HH^{n+1}(E^{i}_t\backslash F_t(K_i))=0,$$
as wished.

Conversely, we want to show $\mathcal{H}^{n+1}(F_t(K_i) \backslash E^i_t)=0$.
Observe that the above argument applied to $K_i^{\prime}:=\overline{\bar{\Omega}\backslash K_i}$  gives $(E^{\prime})^{i}_t\subseteq F_t({K}^{\prime}_i)$ up to a null set. 
We also observe that $E^{i,\varepsilon}_t\cup (E^{\prime})^{i,\varepsilon}_t=\bar{\Omega}\times \RR$, and by further taking $\varepsilon\to 0$ and splitting the $\RR$-factor we infer that $E^i_t\cup (E^{\prime})^i_t=\bar{\Omega}$.
However, by Proposition \ref{prop:intermeasure0} and the nonfattening of $\partial K_i$, we have
\begin{equation}\label{equ:aimofclaim}
\mathcal{H}^{n+1}(F_t(K)\cap F_t(K^{\prime}))\le \mathcal{H}^{n+1}(\KK_t\cap \KK^{\prime}_t)= \mathcal{H}^{n+1}(\MM_t)=0.
\end{equation}
Combining these facts yields $\mathcal{H}^{n+1}(F_t(K_i) \backslash E^i_t)=0$ as desired.
This concludes the proof of the claim.
\end{proof}
Combining Claim \ref{cla:1} and Claim \ref{cla:2}, and since null sets don't affect the reduced boundary, we get
\begin{equation}\label{equ:reducedbdyinmu}
\partial^*_{\Omega}F_t(K_i)\subseteq \spt{\mu^i_t}.
\end{equation}
By further passing to a subsequence in $i$ of $\mu^i_t$, we will obtain a unit-regular and cyclic free boundary Brakke flow $t \mapsto \mu_t$ with $\mu_0=\mu$. Let $\mathcal{S}_i$ and $\mathcal{S}$ be the space-time supports of the flows $t \mapsto \mu^i_t$ and $t \mapsto \mu_t$. This, by definition, implies 
\begin{equation}\label{equ:trivialdef}
\operatorname{spt} \mu^i_t  \subseteq S^i_t.
\end{equation}
 Combining (\ref{equ:reducedbdyinmu}) and (\ref{equ:trivialdef}) and the fact from \cite[Thm 4.4]{giu84}, we get
\begin{equation}\label{equ:usegiu}
\partial (\inte{F_t(K_i)})\cap \Omega \subseteq S^i_t.
\end{equation}
We now conclude similarly as in \cite{hw20}. Let $x \in \partial F_t(K)$ and $\varepsilon>0$. Since $\partial F_t(K)\subset \inte{F_t\left(K_i\right)}$ for all $i$, the ball $B(x, \varepsilon)$ contains a point in $\partial (\inte{F_t\left(K_i\right)}$ and therefore in $S^i_t$ for all large enough $i$. Letting $i \rightarrow \infty$, we see that $\overline{B(x, \varepsilon)}$ contains a point in $S_t$. Since $\varepsilon>0$ is arbitrary, $x \in S_t$. We have shown that $\partial F_t(K_i) \subseteq S_t$ for all $t$. Together with Proposition \ref{prop:bdyflowrelation} (iii) this implies $M_t \subseteq S_t$, i.e.
\begin{equation}\label{equ:outerflowinspt}
\partial \KK\subseteq \mathcal{S}.
\end{equation}
Conversely, by the avoidance principle from Theorem \ref{thm:0elliptic10.5}, we have
$
\mathcal{S} \subseteq \mathcal{K}.
$
Similarly, we have 
$
\mathcal{S}_i \subseteq \mathcal{M}_i,$ and $\mathcal{M}_i\cap \mathcal{K}=\emptyset.
$
Combining them, we get
\begin{equation}\label{equ:hwcpam25-3}
\mathcal{S} \subseteq \partial \mathcal{K}.
\end{equation}
This concludes the proof of the theorem.
\end{proof}

Finally, we can obtain the following corollary from the above proof. The advantange of the corollary is that we don't necessarily have to start from some smooth intial time when matching the flow of some desired density.
\begin{corollary}[restartability with unit density]
Given $\Omega$ convex and any $K\subset \bar{\Omega}$ with $\partial K$ being a smooth compact free boundary surface, and given any $t_* \ge 0$, there exists a unit-regular, cyclic free boundary integral Brakke flow $\MM=\{\mu_t\}_{t\ge t_{*}}$ with $\mu_{t_*}=\HH^n\llcorner M_{t_*}$ and $\overline{\bigcup_{t\ge t_*} \spt \mu_t}=\bigcup_{t\ge t_*} M_t$.
\end{corollary}

\begin{proof}
Note that we cannot directly apply Theorem \ref{thm:hwcpamB6}, since in the theorem we require the initial surface to be smooth, which is not guaranteed in the process.
However, we only require smoothness in the approximation of the intial surface in the proof and we don't require it while doing the elliptic regularization.
Based on this, we can simply take the perturbation $K_i$ of $K$ constructed above, and start the elliptic regularization with initial condition $\partial P^{i,\varepsilon}=\partial F_{t_*}(K_i)\times \{0\}$. The above proof then gives $\MM=\{\mu_t\}_{t\ge t_*}$ with $\mu_{t_*}=\HH^n\llcorner M_{t_*}$ as desired.
\end{proof}

\bibliographystyle{alpha} 
\bibliography{references}
Yueheng Bao, Department Of Mathematics, University Of Toronto,
40 St. George Street, Toronto, On M5S 2E4, Canada\\
email: bao624@math.utoronto.ca
\end{document}